\documentclass[11pt]{article}

\usepackage{latexsym}
\usepackage{amssymb}
\usepackage{graphicx}

\newtheorem{theorem}{\bf Theorem}[section]
\newtheorem{definition}{\bf Definition}[section]
\newtheorem{lemma}{\bf Lemma}[section]

\newtheorem{corollary}{\bf Corollary}[section]
\def\udots{\mathinner{\mkern1mu\raise-1pt\vbox{\kern7pt\hbox{.}}\mkern2mu
    \raise2pt\hbox{.}\mkern2mu\raise5pt\hbox{.}\mkern1mu}}

\begin{document}
\title{{\bf Mixed Boundary Value  Problems of  Semilinear Elliptic PDEs  and BSDEs with Singular Coefficients } }
\author{{Xue Yang  \qquad Tusheng Zhang\thanks{Corresponding author (E-mail: tusheng.zhang@manchester.ac.uk).}} \\
\\
 { School of Mathematics,\quad University of Manchester,} \\
 { Manchester,\quad UK, \quad M13 9PL}
}

\date{}
\maketitle
\begin{abstract}
In this paper, we prove that there exists a unique weak solution to the mixed boundary value problem for a general class of semilinear second order elliptic partial differential equations with singular coefficients. Our approach is probabilistic. The theory of Dirichlet forms and backward stochastic differential equations with singular coefficients and infinite horizon plays a crucial role.
\end{abstract}


\newpage

\section{Introduction}
\setcounter{equation}{0}
In this paper, our aim is to use probabilistic methods to solve the mixed boundary value problem for semilinear second order elliptic partial differential equations (called PDEs for short) of the following form:
\begin{eqnarray}
\label{eq}
\left\{\begin{array}{ll}
{L}u(x)=-F(x,u(x),\nabla u(x)),&\textrm{on $D$}\\
\frac{1}{2}\frac{\partial u}{\partial \gamma}(x)-\widehat{B}\cdot{ n}(x)u(x)=\Phi(x) &\textrm{on $\partial D$ }
\end{array}\right.
\end{eqnarray}
The elliptic operator $L$ is given by :
\begin{eqnarray}
\label{intro.operator L}
L&=&\frac{1}{2}\nabla\cdot(A\nabla)+B\cdot\nabla-\nabla\cdot(\hat{B}\cdot)+Q\\\nonumber
&=&\frac{1}{2}\sum_{i,j=1}^d \frac{\partial}{\partial x_i}\left(a_{ij}(x)\frac{\partial}{\partial x_j}\right)+\sum_{i=1}^d  B_i(x)\frac{\partial}{\partial x_i}-div(\hat{B}\cdot)+Q(x)
\end{eqnarray}
on a d-dimensional smooth bounded Euclidean domain $D$.\\
$A(x)=(a_{ij})_{1\leq i, j \leq d}$: $R^{d}\rightarrow \ R^{d}\otimes R^{d}$ is a smooth, symmetric
matrix-valued function which is uniformly elliptic. That is, there is a constant $\lambda >1$ such that
\begin{eqnarray}
\label{elliptic constant}
\frac{1}{\lambda}I_{d\times d}\leq A(\cdot)\leq \lambda I_{d\times d}.
\end{eqnarray}
Here $B=(B_{1},...,B_{d})$ and $\hat{B}=(\hat{B}_{1},...,\hat{B}_{d})$ : $R^{d}\rightarrow \ R^{d}$ are Borel measurable
functions, which could be singular, and $Q$ is a real-valued Borel measurable function defined on $R^{d}$ such that, for some  $p>\frac{d}{2}$,
$$
I_{D}(|B|^{2}+|\hat{B}|^{2}+|Q|)\in L^{p}(D).
$$
$L$ is rigorously determined by the following quadratic form:
\begin{eqnarray}
\mathcal{Q}(u,v):=(-Lu,v)_{L^{2}(D)}&=&\frac{1}{2}\sum_{i,j}\int_{D}a_{ij}(x)\frac{\partial u}{\partial x_{i}}\frac{\partial v}{\partial x_{j}}dx-\sum_{i}\int_{D}B_{i}(x)\frac{\partial u}{\partial x_{i}}v(x)dx\nonumber\\
&-&\sum_{i}\int_{D}\hat{B}_{i}(x)\frac{\partial v}{\partial x_{i}}u(x)dx-\int_{D}Q(x)u(x)v(x)dx.\nonumber
\end{eqnarray}
Details about the operator $L$ can be found in ${\cite{GT}}$, ${\cite{MR}}$ and $\cite{T}$.\\
The function $F(\cdot,\cdot,\cdot)$ in $(\ref{eq})$ is a nonlinear function defined on $R^{d}\times R\times R^{d}$ and $\Phi(x)$ is a bounded measurable function defined on the boundary $\partial D$ and $\gamma=An$, where
$n$ denotes the inward normal vector field defined on the boundary $\partial D$.
\vskip 0.3cm
To solve the problem (\ref{eq}), it turns out that we need to establish the existence and uniqueness of solutions of backward stochastic differential equations (BSDEs) with singular coefficients and infinite horizon, which is of independent interest.
\vskip 0.3cm
Probabilistic approaches to boundary value problem of second order differential operators have been adopted by many authors and the earliest work went back as early as 1944 in $\cite{K}$. There has been a lot of study on the Dirichlet boundary problem (see \cite{BH}, $\cite{Gernard}$, $\cite{CZ}$,\cite{DP}, $\cite{PEI}$ and $\cite{Z}$). However, there are not many articles on the probabilistic approaches to the Neumann boundary problem.
\vskip 0.3cm
When $A=I$, $B=0$ and $\hat{B}=0$, the following Neumann boundary problem
\begin{eqnarray}
\left\{\begin{array}{ll}
\frac{1}{2}\triangle u(x)+qu(x)=0,&\textrm{on $D$}\\
\frac{1}{2}\frac{\partial u}{\partial n}(x)=\phi(x) &\textrm{on $\partial D$ }
\end{array}\right.\nonumber
\end{eqnarray}
was solved in $\cite{BH}$ and $\cite{PEI}$, which  also gives the solution the following representation:
\begin{eqnarray}
u(x)=E_{x}[\int_{0}^{\infty}e^{\int_{0}^{t}q(B_{u})du}\phi(B_{t})dL^{0}_{t}],\nonumber
\end{eqnarray}
where $(B_{t})_{t>0}$ is the reflecting Brownian motion on the domain $D$ associated with the infinitesimal generator
\begin{eqnarray}
G=\frac{1}{2}\triangle ,\nonumber
\end{eqnarray}
and $L^{0}_{t}$, $t>0$ is the boundary local time satisfying $L^{0}_{t}=\int_{0}^{t}I_{\partial D}(B_{s})dL^{0}_{s}$.
\vskip 0.3cm
But when $\hat{B}\neq 0$, the term $\nabla\cdot(\hat{B}\cdot)$ is just a formal way of writing  because the divergence does not exist as $\hat{B}$ is only a
measurable vector field. It should be interpreted in the distributional sense. For this reason, the term $\nabla\cdot(\hat{B}\cdot)$ can not be handled by Girsanov transform or Feyman-Kac transform.
\vskip 0.3cm
The study of the boundary value problems for the general operator $L$ in the PDE literature (see e.g. \cite{GT}, \cite{T}) was always carried out under the extra condition:
$$-div(\hat{B})+Q(x)\leq 0$$
in the sense of distribution in order to use the maximum principle.
\vskip 0.3cm
When $F=0$, i.e. the linear case, problem $(\ref{eq})$ was studied in $\cite{CZ2}$( see also $\cite{CZ}$ for the Dirichlet boundary problem). The term $\nabla\cdot(\hat{B}\cdot)$ is tackled using the time-reversal of Girsanov transform of the symmetric reflecting diffusion $(\Omega, P_{x}^{0}, X^{0}_{t},t>0)$ associated with the operator
$$
L_{0}=\frac{1}{2}\nabla \cdot (A\nabla).$$
The semigroup $S_{t}$ associated with the operator $L$ has the following representation (see \cite{CFKZ}):
\begin{eqnarray}
S_{t}f(x)=&&E^{0}_{x}[f(X^{0}_{t})\exp(\int_{0}^{t}(A^{-1}B)^{*}(X^{0}_{s})dM^{0}_{s}+(\int_{0}^{t}(A^{-1}\hat {B})^{*}(X^{0}_{s})dM^{0}_{s})\circ \gamma^{0}_{t}\nonumber\\
&&{}-\frac{1}{2}\int_{0}^{t}(B-\hat{B})A^{-1}(B-\hat{B})^{*}(X^{0}_{s})ds                                                                                                                                                                                                                                                                                                                                                                                                                                                                                                                                                                                                                                                                                                                                                                                                                                                                                                                                                                                                                                                                                                                                                                                                                                                                                                                                                                                                                                                                                                                                                                                                               + \int_{0}^{t}Q(X^{0}_{s})ds)],\nonumber
\end{eqnarray}
where $M^{0}$ is the martingale part of the diffusion $X^{0}$ and $\gamma^{0}_{t}$ is the reverse operator.
\vskip 0.3cm
The main purpose of this paper is to study the nonlinear equation $(\ref{eq})$(i.e. $F\neq 0$), which can not be handled by the methods used for the linear case.  Our approach is first to  solve a  backward stochastic differential equation (BSDE) with singular coefficients and infinite horizon to produce a candidate for the solution of the boundary value problem and then to show that the candidate is indeed a solution. The results we obtained for BSDEs with infinite horizon are  of independent interest.
\vskip 0.3cm
We would like to mention that the first results on BSDEs and  probabilistic interpretation
 of solutions of semilinear parabolic PDEs  via BSDEs were obtained by Peng and pardoux  in $\cite{Peng}$, \cite{PP1} and $\cite{PP}$. There the operator $L$ is smooth and the solution is a viscosity solution. We stress that the solutions we considered for PDEs in this paper are Soblev (also called weak) solutions, not viscosity solutions.

 \vskip 0.3cm
In $\cite{Z}$, the corresponding  Dirichlet problem for the semilinear elliptic PDEs:
\begin{eqnarray}
\label{intro.Dirichlet problem}
\left\{\begin{array}{ll}
{L}u(x)=-F(x,u(x),\nabla u(x)),&\textrm{on $D$}\\
u(x)=\Phi(x) &\textrm{on $\partial D$ }
\end{array}\right.
\end{eqnarray}
was solved.
The strategy in \cite{CZ}, $\cite{Z}$ is to transform the general operator $L$ by a kind of h-transform to  an operator of the form: $L_{2}=\frac{1}{2}\nabla(A\nabla)+b\cdot\nabla+q$ which does not have the "bad" term  such as  $\nabla(\hat{B}\cdot)$. This idea  is  used in current paper too.
\vskip 0.3cm
The BSDEs we studied are inspired by the ones in  $\cite{H}$ where the author gave  a probabilistic interpretation of the solution to the following Neumann problem:
\begin{eqnarray}
\left\{\begin{array}{ll}
(\frac{1}{2}\triangle-\nu)u(x)=0,&\textrm{on $D$}\\
\frac{\partial u}{\partial n}=\phi,&\textrm{on $\partial D$ }
\end{array}\right.\nonumber
\end{eqnarray}
\vskip 0.3cm
The content of the  paper as follows. In Section 2, we study the following  BSDEs with infinite horizon:
\begin{eqnarray}
\label{intro.BSDE}
&&dY(t)=-F(X(t),Y(t),Z(t))dt+e^{\int_{0}^{t}{q}(X(u))dt}\Phi(X(s))dL_{t}+\langle Z(t),dM(t)\rangle,\nonumber\\
&&\lim_{t\rightarrow \infty}e^{\int_{0}^{t}d(X(u))du}Y_{t}=0 \quad in \quad L^{2}(\Omega),
\end{eqnarray}
  where $(X(t))_{t>0}$ is the reflecting diffusion associated with an infinitesimal generator of the form: $\mathcal{A}=\frac{1}{2}\nabla(A\nabla)+b\cdot\nabla$,  $M(t)$ is the martingale part of $X(t)$, $L_t$ is the boundary local time of $X$  and $d(\cdot)$ is an appropriate measurable function.
   The existence and uniqueness of an $L^{2}$-solution  $(Y,Z)$ is obtained. \\
In Section 3, we solve the linear PDEs of the form:
\begin{eqnarray}
\label{intro.linear}
\left\{\begin{array}{ll}
\frac{1}{2}\nabla(A\nabla u)(x)+b\cdot\nabla u(x)+qu(x)=F(x),&\textrm{on $D$}\\
\frac{1}{2}\frac{\partial u}{\partial{\gamma}}(x)=\phi(x) &\textrm{on $\partial D$ }.
\end{array}\right.
\end{eqnarray}
under the condition:
$$E_{x_0}[\int_0^{\infty}e^{\int_0^tq(X(u))du}dL_t]<\infty$$
for some $x_0\in \bar{D}$.
Useful estimates for local time and Girsanov density are  proved which will also  be  used in  subsequent sections.\\
In Section 4, we obtain the solution of the semilinear PDE:
\begin{eqnarray}
\label{intro.semilinear}
\left\{\begin{array}{ll}
\frac{1}{2}\nabla(A\nabla u)(x)+b\cdot\nabla u(x)+qu(x)=G(x,u(x),\nabla u(x)),&\textrm{on $D$}\\
\frac{1}{2}\frac{\partial u}{\partial{\gamma}}(x)=\phi(x) &\textrm{on $\partial D$ }.
\end{array}\right.
\end{eqnarray}
 To this end, we first use the solution $(Y_{x}(t),Z_{x}(t))$ of the BSDE $({\ref{intro.BSDE}})$ to produce a candidate $u_{0}(x)=E_{x}[Y_{x}(0)]$ and then find a solution $u$ of  an equation like (\ref{intro.linear}) with a given $F(x):=G(x,u_{0}(x),v_{0}(x))$. Finally we identify $u$ with $u_0$.
In Section 5, we consider the general problem:
\begin{eqnarray}
\label{intro.final equation}
\left\{\begin{array}{ll}
{L}u(x)=-F(x,u(x)),&\textrm{on $D$}\\
\frac{1}{2}\frac{\partial u}{\partial \gamma}(x)-\widehat{B}\cdot{ n}(x)u(x)=\Phi(x) &\textrm{on $\partial D$ }
\end{array}\right..
\end{eqnarray}
 We apply the transformation introduced in $\cite{CZ}$ to transform the problem $(\ref{intro.final equation})$ to a  problem like $(\ref{intro.semilinear})$.   An inverse transformation will yield the solution of the  problem $(\ref{intro.final equation})$ under the condition that the $L^p$ norm of $\hat{B}$ is sufficiently small.\\
To remove some of the restrictions imposed on $\hat{B}$ in Section 5,  in Section 6, we  study the  $L^{1}$-solutions of the BSDEs $(\ref{intro.BSDE})$ under
appropriate conditions. Our approach is inspired  by the one in \cite{BDHPS}. The study of $L^2$-solutions and $L^1$-solutions of the BSDEs (\ref{intro.BSDE}) are carried out in Section 2 and Section 6 separately because the methods used for these two cases are quite different.

\section{BSDEs with Singular Coefficients and Infinity Horizon}
\setcounter{equation}{0}
Consider the operator
\begin{eqnarray}
L_{1}=\frac{1}{2}\sum_{i,j=1}^d \frac{\partial}{\partial x_i}\left(a_{ij}(x)\frac{\partial}{\partial x_j}\right)+\sum_{i=1}^d  b_i(x)\frac{\partial}{\partial x_i}\nonumber
\end{eqnarray}
on the domian D equipped with the Neumann boundary condition:
\begin{eqnarray}
\frac{\partial}{\partial \gamma}:=\langle An,\nabla\cdot\rangle=0,\quad on \quad \partial D.\nonumber
\end{eqnarray}
By $\cite{LS}$, there exists a unique reflecting diffusion process denoted by $(\Omega, \mathcal{F}_{t}, X_{x}(t), P_{x},\theta_{t}, x\in D)$ associated with the generator $L_{1}$.\\
Here $\theta:\Omega\rightarrow \Omega$ is the shift operator defined as follows:
$$
X_{x}(s)(\theta_{t}\cdot)=X_{x}(t+s),\quad s,t\geq 0.
$$
Let $E_{x}$ denote the expectation under the measure $P_{x}$.\\
Set $\tilde{b}=\{\tilde{b}_{1},...,\tilde{b}_{d}\}$, where  $\tilde{b}_{i}=\frac{1}{2}\sum_{j}\frac{\partial a_{ij}}{\partial x_{j}}+b_{i}$.\\
Then the process $X_{x}(t)$ has the following decomposition:
\begin{eqnarray}
\label{deomposition of X}
X_{x}(t)=X_{x}(0)+M_{x}(t)+\int_{0}^{t}\tilde{b}(X_{x}(s))ds+\int_{0}^{t}An(X_{x}(s))dL_{s},\quad P_{x}-a.s..
\end{eqnarray}
Here $M_{x}(t)$ is a $\mathcal{F}_{t}$ square integrable continuous martingale additive functional. And $L_{t}$
is a positive increasing continuous additive functional satisfying $L_{t}=\int_{0}^{t}I_{\{X_{x}(s)\in\partial D\}}dL_{s}$.\\
We write $X_{x}(t)$ as $X(t)$ for short in the following discussion.\\\\
In this section, we will study the backward stochastic differential equations with singular coefficients and infinite horizon associated with the martingale part $M_{x}(t)$ and the local time $L_{t}$. A unique $L^{2}$ solution of such  BSDEs is obtained.\\\\
Let $g(\omega,t,y,z):\Omega\times R^{+}\times R\times R^{d}\rightarrow R$ be a progressively measurable function. Consider the following conditions:\\
\textbf{(A.1)} $(y_{1}-y_{2})(g(t,y_{1},z)-g(t,y_{2},z))\leq -a_{1}(t)|y_{1}-y_{2}|^{2}$,\\
\textbf{(A.2)} $|g(t,y,z_{1})-g(t,y,z_{2})|\leq a_{2}|z_{1}-z_{2}|$,\\
\textbf{(A.3)} $|g(t,y,z)|\leq|g(t,0,0)|+ a_{3}(t)(1+|y|)$.\\
Here $a_{1}(t)$ and $a_{3}(t)$ are two progressively measurable processes and $a_{2}$ is a constant.\\
Set $a(t)=-a_{1}(t)+\delta a^{2}_{2}$, for some constant $\delta>\frac{1}{2\lambda}$, where $\lambda $ is the constant appeared in $(\ref{elliptic constant})$.\\
\begin{lemma}
\label{lemma basic BSDE}
Assume the conditions (A.1)-(A.3) and
$$
E_{x}[\int_{0}^{\infty}e^{2\int_{0}^{t}a(u)du}|g(t,0,0)|^{2}dt]<\infty.
$$
Then there exists a unique solution $(Y_{x}(t),Z_{x}(t))$ to the following backward stochastic differential equation:
\begin{eqnarray}
\label{0 terminal value equation}
&&Y_{x}(t)=Y_{x}(T)+\int_{t}^{T}g(s,Y_{x}(s),Z_{x}(s))ds-\int_{t}^{T}<Z_{x}(s),dM_{x}(s)>, \quad t<T;\nonumber\\
&&\lim_{t\rightarrow\infty}e^{\int_{0}^{t}a(u)du}Y_{x}(t)=0, \quad in \quad L^{2}(\Omega).
\end{eqnarray}
Moreover,
\begin{eqnarray}
\label{estimates basic BSDE}
E_{x}[\sup_{t}e^{2\int_{0}^{t}a(u)du}|Y_{x}(t)|^{2}]<\infty \quad and \quad E_{x}[\int_{0}^{\infty}e^{2\int_{0}^{s}a(u)du}|Z_{x}(s)|^{2}ds]<\infty.
\end{eqnarray}
\end{lemma}
\begin{proof}\\
Existence:\\
The proof of this lemma is similar to that of Theorem 3.2 in $\cite{Z}$, but the terminal conditions here are different. By Theorem 3.1 in $\cite{Z}$, the following BSDE has a unique solution $(Y^{n}_{x}(t),Z^{n}_{x}(t))$:
\begin{eqnarray}
\label{BSDE Yn}
Y^{n}_{x}(t)=\int_{t}^{n}g(s,Y^{n}_{x}(s),Z^{n}_{x}(s))ds-\int_{t}^{n}<Z^{n}_{x}(s),dM_{x}(s)>,\quad t\leq n;
\end{eqnarray}
and moreover,
\begin{eqnarray}
Y^{n}_{x}(t)=0,\quad Z^{n}_{x}(t)=0,\quad t>n.\nonumber
\end{eqnarray}
Fix $t>0$ and $n>m>t$. It follows that
\begin{eqnarray}
&&e^{2\int_{0}^{t}a(u)du}|Y^{n}_{x}(t)-Y^{m}_{x}(t)|^{2}+\int_{t}^{\infty}e^{2\int_{0}^{s}a(u)du}\langle A(X(s))(Z^{n}_{x}(s)-Z^{m}_{x}(s)),(Z^{n}_{x}(s)-Z^{m}_{x}(s))\rangle ds\nonumber\\
&=& -2\int_{t}^{n}a(s)e^{2\int_{0}^{s}a(u)du}|Y^{n}_{x}(s)-Y^{m}_{x}(s)|^{2}ds\nonumber\\
&&+2\int_{t}^{n}e^{2\int_{0}^{s}a(u)du}(Y^{n}_{x}(s)-Y^{m}_{x}(s))(g(s,Y^{n}_{x}(s),Z^{n}_{x}(s))
-g(s,Y^{m}_{x}(s),Z^{m}_{x}(s)))ds\nonumber\\
&&+2\int_{m}^{n}e^{2\int_{0}^{s}a(u)du}(Y^{n}_{x}(s)-Y^{m}_{x}(s))g(s,0,0)ds\nonumber\\
&&-2\int_{t}^{n}e^{2\int_{0}^{s}a(u)du}(Y^{n}_{x}(s)-Y^{m}_{x}(s))<Z^{n}_{x}(s)-Z^{m}_{x}(s),dM_{x}(t)>\nonumber
\end{eqnarray}
Choose two positive numbers $\delta_{1}$ and $\delta_{2}$ such that $\delta_{1}>\frac{1}{2\lambda}$ and
$\delta_{1}+\delta_{2}<\delta$.
Then  from
\begin{eqnarray}
&&2\int_{t}^{n}e^{2\int_{0}^{s}a(u)du}(Y^{n}_{x}(s)-Y^{m}_{x}(s))(g(s,Y^{n}_{x}(s),Z^{n}_{x}(s))
-g(s,Y^{m}_{x}(s),Z^{m}_{x}(s)))ds\nonumber\\
&\leq&-2\int_{t}^{n}a_{1}(s)e^{2\int_{0}^{s}a(u)du}|Y^{n}_{x}(s)-Y^{m}_{x}(s)|^{2}ds\nonumber\\
&&+2\delta_{1}a_{2}^{2}\int_{t}^{n}e^{2\int_{0}^{s}a(u)du}|Y^{n}_{x}(s)-Y^{m}_{x}(s)|^{2}ds\nonumber\\
&&+\frac{1}{2\lambda\delta_{1}}\int_{t}^{n}e^{2\int_{0}^{s}a(u)du}\langle A(X(s))(Z^{n}_{x}(s)-Z^{m}_{x}(s)),(Z^{n}_{x}(s)-Z^{m}_{x}(s))\rangle ds\nonumber
\end{eqnarray}
and
\begin{eqnarray}
&&2\int_{m}^{n}e^{2\int_{0}^{s}a(u)du}(Y^{n}_{x}(s)-Y^{m}_{x}(s))g(s,0,0)ds\nonumber\\
&\leq&2\delta_{2}a_{2}^{2}\int_{m}^{n}e^{2\int_{0}^{s}a(u)du}|Y^{n}_{x}(s)-Y^{m}_{x}(s)|^{2}ds\nonumber
+\frac{1}{2\delta_{2}a_{2}^{2}}\int_{m}^{n}e^{2\int_{0}^{s}a(u)du}|g(s,0,0)|^{2}ds,\nonumber
\end{eqnarray}
it follows that
\begin{eqnarray*}
&&E_{x}[e^{2\int_{0}^{t}a(u)du}|Y^{n}_{x}(t)-Y^{m}_{x}(t)|^{2}]
+\frac{1}{\lambda}(1-\frac{1}{2\lambda\delta_{1}})E_{x}[\int_{t}^{\infty}e^{2\int_{0}^{s}a(u)du}|Z^{n}_{x}(s)-Z^{m}_{x}(s)|^{2}ds]\nonumber\\
&&\leq\frac{1}{2\delta_{2}a_{2}^{2}}E_{x}[\int_{m}^{n}e^{2\int_{0}^{s}a(u)du}|g(s,0,0)|^{2}ds].
\end{eqnarray*}
This implies that
\begin{eqnarray}
&&E_{x}[\int_{0}^{\infty}e^{2\int_{0}^{s}a(u)du}|Z^{n}_{x}(s)-Z^{m}_{x}(s)|^{2}ds]\rightarrow0,\quad as \quad m,n\rightarrow\infty.\nonumber
\end{eqnarray}
Hence there exists $\tilde{Z}_{x}$ such that
\begin{eqnarray}
\tilde{Z}_{x}=\lim_{n\rightarrow\infty}e^{\int_{0}^{\cdot}a(u)du}Z^{n}_{x}
\quad in \quad L^{2}([0,\infty)\times\Omega).\nonumber
\end{eqnarray}
At the same time,  we also obtain the following estimates:
\begin{eqnarray}
&&\sup_{t}e^{2\int_{0}^{t}a(u)du}|Y^{n}_{x}(t)-Y^{m}_{x}(t)|^{2}\nonumber\\
&\leq& \frac{1}{2\delta_{2}a_{2}^{2}}\int_{m}^{n}e^{2\int_{0}^{s}a(u)du}|g(s,0,0)|^{2}ds \nonumber\\
&&+2\sup_{t}|\int_{t}^{n}e^{2\int_{0}^{s}a(u)du}(Y^{n}_{x}(s)-Y^{m}_{x}(s))<Z^{n}_{x}(s)-Z^{m}_{x}(s),dM_{x}(t)>|.\nonumber
\end{eqnarray}
Taking expectation on both sides of the above inequality, by BDG inequality, we obtain
\begin{eqnarray}
&&E_{x}[\sup_{t}e^{2\int_{0}^{t}a(u)du}|Y^{n}_{x}(t)-Y^{m}_{x}(t)|^{2}]\nonumber\\
&\leq&\frac{1}{2\delta_{2}a_{2}^{2}}E_{x}[\int_{m}^{n}e^{2\int_{0}^{s}a(u)du}|g(s,0,0)|^{2}ds]\nonumber\\
&&+C_{1}E_{x}[\{\int_{t}^{n}e^{4\int_{0}^{s}a(u)du}|Y^{n}_{x}(s)-Y^{m}_{x}(s)|^{2}|Z^{n}_{x}(s)-Z^{m}_{x}(s)|^{2}ds\}^{\frac{1}{2}}]\nonumber\\
&\leq& \frac{1}{2\delta_{2}a_{2}^{2}}E_{x}[\int_{m}^{n}e^{2\int_{0}^{s}a(u)du}|g(s,0,0)|^{2}ds]
+ \frac{1}{2}E_{x}[\sup_{t}e^{2\int_{0}^{t}a(u)du}|Y^{n}_{x}(t)-Y^{m}_{x}(t)|^{2}]\nonumber\\
&&+C_{2}E_{x}[\int_{0}^{\infty}e^{2\int_{0}^{s}a(u)du}|Z^{n}_{x}(s)-Z^{m}_{x}(s)|^{2}ds]\nonumber
\end{eqnarray}
Thus
\begin{eqnarray}
&&E_{x}[\sup_{t}e^{2\int_{0}^{t}a(u)du}|Y^{n}_{x}(t)-Y^{m}_{x}(t)|^{2}]\nonumber\\
&\leq& \frac{1}{\delta_{2}a_{2}^{2}}E_{x}[\int_{m}^{n}e^{2\int_{0}^{s}a(u)du}|g(s,0,0)|^{2}ds]+
2C_{2}E_{x}[\int_{0}^{\infty}e^{2\int_{0}^{s}a(u)du}|Z^{n}_{x}(s)-Z^{m}_{x}(s)|^{2}ds]\nonumber\\
&\rightarrow& 0,\quad as\quad m,n\rightarrow\infty.\nonumber
\end{eqnarray}
So, there exists $\{\tilde{Y}_{x}(t)\}$ such that
\begin{eqnarray}
\lim_{n\rightarrow\infty}E_{x}[\sup_{t}|\tilde{Y}_{x}(t)-e^{\int_{0}^{t}a(u)du}Y^{n}_{x}(t)|^{2}]=0.\nonumber
\end{eqnarray}
For any $\varepsilon>0$, there exist a positive number $N$ such that for any $n\geq N$,
$$
E_{x}[\sup_{t}|\tilde{Y}_{x}(t)-e^{\int_{0}^{t}a(u)du}Y^{n}_{x}(t)|^{2}]<\frac{\varepsilon}{2}.
$$
For $t>N$, noticing $Y^{N}_{x}(t)=0$,  it follows that
\begin{eqnarray}
E_{x}[|\tilde{Y}_{x}(t)|^{2}]&\leq& 2E_{x}[|\tilde{Y}_{x}(t)-e^{\int_{0}^{t}a(u)du}|Y^{N}_{x}(t)|^{2}]+2E_{x}[e^{2\int_{0}^{t}a(u)du}|Y^{N}_{x}(t)|^{2}]\nonumber\\
&\leq&2E_{x}[\sup_{t}|\tilde{Y}_{x}(t)-e^{\int_{0}^{t}a(u)du}|Y^{N}_{x}(t)|^{2}]+2E_{x}[e^{2\int_{0}^{t}a(u)du}|Y^{N}_{x}(t)|^{2}]\nonumber\\
&<&\varepsilon.\nonumber
\end{eqnarray}
Thus we have $\displaystyle{\lim_{t\rightarrow 0}}E_{x}[|\tilde{Y}_{x}(t)|^{2}]=0$.\\
By chain rule, it is easy to see from $(\ref{BSDE Yn})$ that
$$
{Y}_{x}(t)=e^{-\int_{0}^{t}a(u)du}\tilde{Y}_{x}(t) \quad and \quad {Z}_{x}(t)=e^{-\int_{0}^{t}a(u)du}\tilde{Z}_{x}(t)
$$
satisfy the equation $(\ref{0 terminal value equation})$
and
$$
\lim_{t\rightarrow \infty}E_{x}[e^{2\int_{0}^{t}a(u)du}|Y_{x}(t)|^{2}]=\lim_{t\rightarrow \infty}E_{x}[|\tilde{Y}_{x}(t)|^{2}]=0.
$$
From the above proof, we also see that $(\ref{estimates basic BSDE})$ holds.\\
Uniqueness:\\
Suppose that $(Y^{1}_{x},Z^{1}_{x})$ and $(Y^{2}_{x},Z^{2}_{x})$ are two solutions of  the equation $(\ref{0 terminal value equation})$.\\
Set $\bar{Y}_{x}(t)=Y^{1}_{x}(t)-Y^{2}_{x}(t)$ and $\bar{Z}_{x}(t)=Z^{1}_{x}(t)-Z^{2}_{x}(t)$ . Then
\begin{eqnarray}
d(e^{\int_{0}^{t}a(u)du}\bar{Y}_{x}(t))
&=&-e^{\int_{0}^{t}a(u)du}(g(t,Y^{1}_{x}(t),Z^{1}_{x}(t))-g(t,Y^{2}_{x}(t),Z^{2}_{x}(t)))dt\nonumber\\
&+&a(t)e^{\int_{0}^{t}a(u)du}\bar{Y}_{x}(t)dt\nonumber\\
&+& e^{\int_{0}^{t}a(u)du}\langle\bar{Z}_{x}(t),dM_{x}(t)\rangle.
\end{eqnarray}
 By Ito's formula, we get, for any $t<T$,
\begin{eqnarray}
&&e^{2\int_{0}^{t}a(u)du}|\bar{Y}_{x}(t)|^{2}+\int_{t}^{T}e^{2\int_{0}^{t}a(u)du}\langle A(X(s))\bar {Z}_{x}(s),\bar {Z}_{x}(s)\rangle ds\nonumber\\
&=&e^{2\int_{0}^{T}a(u)du}|\bar{Y}_{x}(T)|^{2}+2\int_{t}^{T}e^{2\int_{0}^{s}a(u)du}\bar{Y}_{x}(s)(g(s,Y^{1}_{x}(s),Z^{1}_{x}(s))-g(s,Y^{2}_{x}(s),Z^{2}_{x}(s)))ds\nonumber\\
&&-2\int_{t}^{T}a(s)e^{2\int_{0}^{s}a(u)du}|\bar{Y}_{x}(s)|^{2}ds\nonumber\\
&&-2\int_{t}^{T}a(s)e^{2\int_{0}^{s}a(u)du}\bar{Y}_{x}(s)\langle\bar{Z}_{x}(s),dM_{x}(s)\rangle
\end{eqnarray}
By condition (A.1) and (A.2), we have
\begin{eqnarray}
&&2\int_{t}^{T}e^{2\int_{0}^{s}a(u)du}\bar{Y}_{x}(s)(g(s,Y^{1}_{x}(s),Z^{1}_{x}(s))-g(s,Y^{2}_{x}(s),Z^{2}_{x}(s)))ds\nonumber\\
&=&2\int_{t}^{T}e^{2\int_{0}^{s}a(u)du}\bar{Y}_{x}(s)(g(s,Y^{1}_{x}(s),Z^{1}_{x}(s))-g(s,Y^{2}_{x}(s),Z^{1}_{x}(s)))ds\nonumber\\
&+&2\int_{t}^{T}e^{2\int_{0}^{s}a(u)du}\bar{Y}_{x}(s)(g(s,Y^{2}_{x}(s),Z^{1}_{x}(s))-g(s,Y^{2}_{x}(s),Z^{2}_{x}(s)))ds\nonumber\\
&\leq& -2\int_{t}^{T}a_{1}(s)e^{2\int_{0}^{s}a(u)du}|\bar{Y}_{x}(s)|^{2}ds
+a_{2}\int_{t}^{T}e^{2\int_{0}^{s}a(u)du}\bar{Y}_{x}(s)|\bar{Z}_{x}(s)|ds\nonumber\\
&\leq& -2\int_{t}^{T}a_{1}(s)e^{2\int_{0}^{s}a(u)du}|\bar{Y}_{x}(s)|^{2}ds
+c'a_{2}\int_{t}^{T}e^{2\int_{0}^{s}a(u)du}|\bar{Y}_{x}(s)|^{2}ds\nonumber\\
&&+a_{2}\frac{1}{c'\lambda}\int_{t}^{T}e^{2\int_{0}^{s}a(u)du}|\bar{Z}_{x}(s)|^{2}ds.
\end{eqnarray}
Choosing $c'=2\delta a_{2}$, we obtain
\begin{eqnarray}
&&|e^{\int_{0}^{t}a(u)du}\bar{Y}_{x}(t)|^{2}+(1-\frac{1}{2\delta\lambda})\int_{t}^{T}e^{2\int_{0}^{t}a(u)du}\langle A(X(s))\bar {Z}_{x}(s),\bar {Z}_{x}(s)\rangle ds\nonumber\\
&\leq& e^{2\int_{0}^{T}a(u)du}|\bar{Y}_{x}(T)|^{2}-2\int_{t}^{T}a(s)e^{2\int_{0}^{s}a(u)du}\bar{Y}_{x}(s)\langle\bar{Z}_{x}(s),dM_{x}(s)\rangle
\end{eqnarray}
Taking expectation on both sides of the above inequality, we get that, for any $t<T$,
$$
E_{x}[e^{2\int_{0}^{t}a(u)du}|\bar{Y}_{x}(t)|^{2}]\leq E_{x}[e^{2\int_{0}^{T}a(u)du}|\bar{Y}_{x}(T)|^{2}].
$$
For both $Y^{1}$ and $Y^{2}$ satisfy the terminal condition in $(\ref{0 terminal value equation})$, so that
$$
\lim_{T\rightarrow \infty}E_{x}[e^{2\int_{0}^{T}a(u)du}|\bar{Y}_{x}(T)|^{2}]=0,
$$
which leads to $E_{x}[e^{2\int_{0}^{t}a(u)du}|\bar{Y}_{x}(t)|^{2}]=0$.\\
We conclude that $Y^{1}_{x}(t)=Y^{2}_{x}(t)$ and $Z^{1}_{x}(t)=Z^{2}_{x}(t)$.  $\square$
\end{proof}
\\\\

We now want to apply Lemma $\ref{lemma basic BSDE}$ to a particular situation.\\
\noindent{Let} $F(x,y,z):R^{d}\times R\times R^{d}\rightarrow R$ be a Borel measurable function. Consider the  following conditions:\\
\textbf{(D.1)} $(y_{1}-y_{2})(F(x,y_{1},z)-F(x,y_{2},z))\leq -d_{1}(x)|y_{1}-y_{2}|^{2}$,\\
\textbf{(D.2)} $|F(x,y,z_{1})-F(x,y,z_{2})|\leq d_{2}|z_{1}-z_{2}|$,\\
\textbf{(D.3)} $|F(x,y,z)|\leq |F(x,0,z)|+ K(x)(1+|y|)$.\\
Set $d(x)=-d_{1}(x)+\delta d^{2}_{2}$ for some constant $\delta>\frac{1}{2\lambda}$.\\
The follows result follows from Lemma $\ref{lemma basic BSDE}$.
\begin{lemma}
Assume the conditions (D.1)-(D.3)  and
$$
E_{x}[\int_{0}^{\infty}e^{2\int_{0}^{t}d(X(u))du}|F(X(t),0,0)|^{2}dt]<\infty.
$$
Then there exists a unique solution $(Y_{x}(t),Z_{x}(t))$ to the following equation:
\begin{eqnarray}
\label{BSDE infinity horizon}
&&Y_{x}(t)=Y_{x}(T)+\int_{t}^{T}F(X(s),Y_{x}(s),Z_{x}(s))ds-\int_{t}^{T}<Z_{x}(s),dM_{x}(s)>,\quad t<T;\nonumber\\
&&\lim_{t\rightarrow\infty}e^{\int_{0}^{t}d(X(u))du}Y_{x}(t)=0,\quad in\quad L^{2}(\Omega).
\end{eqnarray}
\end{lemma}

\vspace{4mm}
Consider the following condition instead of $(D.3)$.\\
$\textbf{(D.3)}'$  $|F(X(t),y,z)|\leq K(t)$, for any $y\in R$ and $z\in R^{d}$.
\vskip 0.3cm
Let $\Phi$ be a bounded measurable function defined on $\partial D$, and function $\tilde{q}\in L^{p}(D)$, for $p>\frac{d}{2}$. \\
The following theorem is the main result in this section.
\begin{theorem}
\label{thm 1}
Assume the conditions (D.1), (D.2) and (D.3)',
\begin{eqnarray}
\label{theorem: main bsde int.condition}
E_{x_{0}}[\int_{0}^{\infty}e^{\int_{0}^{s}\tilde{q}(X(u))du}dL_{s}]<\infty\nonumber
\end{eqnarray}
for some $x_{0}\in D$ and for $x\in D$,
\begin{eqnarray}
\label{integrability}
E_{x}[\int_{0}^{\infty}e^{2\int_{0}^{t}d(X(u))du}\{e^{2\int_{0}^{t}\tilde{q}(X(u))du}+|K(t)|^{2}\}dt]<\infty.
\end{eqnarray}
Then there exists a unique solution $(Y_{x},Z_{x})$ to the following BSDE:
\begin{eqnarray}
\label{bsde}
Y_{x}(t)&=&Y_{x}(T)+\int_{t}^{T}F(X(s),Y_{x}(s),Z_{x}(s))ds-\int_{t}^{T}e^{\int_{0}^{s}\tilde{q}(X(u))dt}\Phi(X(s))dL_{s}\nonumber\\
&&-\int_{t}^{T}\langle Z_{x}(s),dM_{x}(s)\rangle,\quad for \quad t<T,
\end{eqnarray}
and
\begin{eqnarray}
\label{uniqueness condition}
\lim_{t\rightarrow \infty}e^{\int_{0}^{t}d(X(u))du}Y_{t}=0 \quad in \quad L^{2}(\Omega).
\end{eqnarray}
\end{theorem}
\begin{proof}\\
Uniqueness:\\
Suppose that $(Y^{1}_{x},Z^{1}_{x})$ and $(Y^{2}_{x},Z^{2}_{x})$ are two solutions of  the equation $(\ref{bsde})$ satisfying $(\ref{uniqueness condition})$.\\
Set $\bar{Y}_{x}(t)=Y^{1}_{x}(t)-Y^{2}_{x}(t)$ and $\bar{Z}_{x}(t)=Z^{1}_{x}(t)-Z^{2}_{x}(t)$ . Then
\begin{eqnarray}
d(e^{\int_{0}^{t}d(X(u))du}\bar{Y}_{x}(t))
&=&-e^{\int_{0}^{t}d(X(u))du}(F(X(t),Y^{1}_{x}(t),Z^{1}_{x}(t))-F(X(t),Y^{2}_{x}(t),Z^{2}_{x}(t)))dt\nonumber\\
&+&d(X(t))e^{\int_{0}^{t}d(X(u))du}\bar{Y}_{x}(t)dt\nonumber\\
&+& e^{\int_{0}^{t}d(X(u))du}\langle\bar{Z}_{x}(t),dM_{x}(t)\rangle.
\end{eqnarray}
By Ito's formula, we get, for any $t<T$,
\begin{eqnarray}
&&e^{2\int_{0}^{t}d(X(u))du}|\bar{Y}_{x}(t)|^{2}+\int_{t}^{T}e^{2\int_{0}^{t}d(X(u))du}\langle A(X(s))\bar {Z}_{x}(s),\bar {Z}_{x}(s)\rangle ds\nonumber\\
&=&e^{2\int_{0}^{T}d(X(u))du}|\bar{Y}_{x}(T)|^{2}\nonumber\\
&+&2\int_{t}^{T}e^{2\int_{0}^{s}d(X(u))du}\bar{Y}_{x}(s)(F(X(s),Y^{1}_{x}(s),Z^{1}_{x}(s))-F(X(s),Y^{2}_{x}(s),Z^{2}_{x}(s)))ds\nonumber\\
&-&2\int_{t}^{T}d(X(s))e^{2\int_{0}^{s}d(X(u))du}|\bar{Y}_{x}(s)|^{2}ds\nonumber\\
&-&2\int_{t}^{T}d(X(s))e^{2\int_{0}^{s}d(X(u))du}\bar{Y}_{x}(s)\langle\bar{Z}_{x}(s),dM_{x}(s)\rangle
\end{eqnarray}
By (D.1) and (D.2), we have
\begin{eqnarray}
\label{uniqueness formula}
&&2\int_{t}^{T}e^{2\int_{0}^{s}d(X(u))du}\bar{Y}_{x}(s)(F(X(s),Y^{1}_{x}(s),Z^{1}_{x}(s))-F(X(s),Y^{2}_{x}(s),Z^{2}_{x}(s)))ds\nonumber\\
&=&2\int_{t}^{T}e^{2\int_{0}^{s}d(X(u))du}\bar{Y}_{x}(s)(F(X(s),Y^{1}_{x}(s),Z^{1}_{x}(s))-F(X(s),Y^{2}_{x}(s),Z^{1}_{x}(s)))ds\nonumber\\
&+&2\int_{t}^{T}e^{2\int_{0}^{s}d(X(u))du}\bar{Y}_{x}(s)(F(X(s),Y^{2}_{x}(s),Z^{1}_{x}(s))-F(X(s),Y^{2}_{x}(s),Z^{2}_{x}(s)))ds\nonumber\\
&\leq& -2\int_{t}^{T}d_{1}(X(s))e^{2\int_{0}^{s}d(X(u))du}|\bar{Y}_{x}(s)|^{2}ds
+d_{2}\int_{t}^{\infty}e^{2\int_{0}^{s}d(X(u))du}\bar{Y}_{x}(s)|\bar{Z}_{x}(s)|ds\nonumber\\
&\leq& -2\int_{t}^{T}d_{1}(X(s))e^{2\int_{0}^{s}d(X(u))du}|\bar{Y}_{x}(s)|^{2}ds
+cd_{2}\int_{t}^{T}e^{2\int_{0}^{s}d(X(u))du}|\bar{Y}_{x}(s)|^{2}ds\nonumber\\
&&+d_{2}\frac{1}{c\lambda}\int_{t}^{T}e^{2\int_{0}^{s}d(X(u))du}|\bar{Z}_{x}(s)|^{2}ds.
\end{eqnarray}
Choosing $c=2\delta d_{2}$, we obtain from $(\ref{uniqueness formula})$
\begin{eqnarray}
\label{uniqueness of Y Z}
&&|e^{-\int_{0}^{t}d(X(u))du}\bar{Y}_{x}(t)|^{2}+(1-\frac{1}{2\delta\lambda})\int_{t}^{T}e^{-2\int_{0}^{t}d(X(u))du}\langle A(X(s))\bar {Z}_{x}(s),\bar {Z}_{x}(s)\rangle ds\nonumber\\
&\leq& e^{2\int_{0}^{T}d(X(u))du}|\bar{Y}_{x}(T)|^{2} -2\int_{t}^{T}d(X(s))e^{-2\int_{0}^{s}d(X(u))du}\bar{Y}_{x}(s)\langle\bar{Z}_{x}(s),dM_{x}(s)\rangle
\end{eqnarray}
Taking expectation on both sides of the above inequality and letting T tend to infinity, we obtain that
$$
E_{x}[e^{2\int_{0}^{t}d(X(u))du}|\bar{Y}_{x}(t)|^{2}]=0
$$
We conclude that $Y^{1}_{x}(t)=Y^{2}_{x}(t)$ and hence from $(\ref{uniqueness of Y Z})$, $Z^{1}_{x}(t)=Z^{2}_{x}(t)$.\\\\
Existence:\\
First of all, the assumption $(\ref{theorem: main bsde int.condition})$ implies (see \cite{CZ1})
\begin{eqnarray*}
\sup_{x}E_{x}[\int_{0}^{\infty}e^{\int_{0}^{s}\tilde{q}(X(u))du}dL_{s}]<\infty.
\end{eqnarray*}
$1^{\circ}$: There exists $(p_{x}(t),q_{x}(t))$  such that
\begin{eqnarray}
\label{pq}
dp_{x}(t)=e^{\int_{0}^{t}\tilde{q}(X(u))du}\Phi(X(t))dL_{t}+<q_{x}(t),dM_{x}(t)>,
\end{eqnarray}
and $e^{\int_{0}^{t}d(X(u))du}p_{x}(t)\rightarrow 0$ as $t\rightarrow\infty$ in $L^{2}(\Omega)$.\\
In fact, let
\begin{eqnarray} p_{x}(t)&:=&-E_{x}[\int_{t}^{\infty}e^{\int_{0}^{s}\tilde{q}(X(u))du}\Phi(X(s))dL_{s}|\mathcal{F}_{t}]\nonumber\\
&=& \int_{0}^{t}e^{\int_{0}^{s}\tilde{q}(X(u))du}\Phi(X(s))L_{s}
-E_{x}[\int_{0}^{\infty}e^{\int_{0}^{s}\tilde{q}(X(u))du}\Phi(X(s))dL_{s}|\mathcal{F}_{t}].\nonumber\\
\end{eqnarray}
By the martingale representation theorem in $\cite{Z}$, there exists a process $q_{x}(t)$, such that
\begin{eqnarray}
\label{integrability of q}
-E_{x}[\int_{0}^{\infty}e^{\int_{0}^{s}\tilde{q}(X(u))du}\Phi(X(s))dL_{s}|\mathcal{F}_{t}]
=&-&E_{x}[\int_{0}^{\infty}e^{\int_{0}^{s}\tilde{q}(X(u))du}\Phi(X(s))dL_{s}]\nonumber\\
&+& \int_{0}^{t}<q_{x}(s),dM_{x}(s)>.
\end{eqnarray}
Then $(p_{x},q_{x})$ satisfies the equation $(\ref{pq})$.\\
Moreover,
\begin{eqnarray}
\label{formula of px}
p_{x}(t)&:=&-E_{x}[\int_{t}^{\infty}e^{\int_{0}^{s}\tilde{q}(X(u))du}\Phi(X(s))dL_{s}|\mathcal{F}_{t}]\nonumber\\
&=&-e^{\int_{0}^{t}\tilde{q}(X(u))du}E_{x}[\int_{t}^{\infty}e^{\int_{t}^{s}\tilde{q}(X(u))du}\Phi(X(s))dL_{s}|\mathcal{F}_{t}]\nonumber\\
&=&-e^{\int_{0}^{t}\tilde{q}(X(u))du}E_{x}[\int_{0}^{\infty}e^{\int_{t}^{s+t}\tilde{q}(X(u))du}\Phi(X(s+t))dL_{s+t}|\mathcal{F}_{t}]\nonumber\\
&=&-e^{\int_{0}^{t}\tilde{q}(X(u))du}E_{x}[\int_{0}^{\infty}e^{\int_{0}^{s}\tilde{q}(X(u+t))du}\Phi(X(s+t))dL_{s+t}|\mathcal{F}_{t}]\nonumber\\
&=&-e^{\int_{0}^{t}\tilde{q}(X(u))du}E_{X(t)}[\int_{0}^{\infty}e^{\int_{0}^{l}\tilde{q}(X(u))du}\Phi(X(l))dL_{l}]
\end{eqnarray}
The last equality follows from the fact that $L_{t+s}=L_{t}+L_{s}\circ \theta_{t}$. Therefore,
$$
\sup_{x}|p_{x}(t)|\leq e^{\int_{0}^{t}\tilde{q}(X(u))du} \sup_{x\in D}|\Phi(x)|\cdot\sup_{x\in \bar D}E_{x}[\int_{0}^{\infty}e^{\int_{0}^{t}\tilde{q}(X(u))du}dL_{t}].
$$
Set $\displaystyle{M=\sup_{x\in D}|\Phi(x)|\cdot\sup_{x\in \bar D}E_{x}[\int_{0}^{\infty}e^{\int_{0}^{t}\tilde{q}(X(u))du}dL_{t}]}$.\\
In view of $(\ref{integrability})$, we have $\displaystyle{\lim_{t\rightarrow \infty}}e^{\int_{0}^{t}(d+\tilde{q})(X(u))du}=0$ in $L^{2}(\Omega)$.\\
Hence,
\begin{eqnarray}
e^{\int_{0}^{t}d(X(u))du}p_{x}(t)\leq Me^{\int_{0}^{t}(d+\tilde{q})(X(u))du}\rightarrow 0 \quad as \quad t\rightarrow\infty, \quad in \quad L^{2}(\Omega).
\end{eqnarray}
$2^{\circ}:$  Set $g(t,y,z)=F(X(t),p_{x}(t)+y,q_{x}+z)$. Then
\begin{eqnarray}
\label{condition 1}
&&(y_{1}-y_{2})(g(t,y_{1},z)-g(t,y_{2},z))\nonumber\\
&=&(y_{1}-y_{2})(F(X(t),p_{x}(t)+y_{1},q_{x}+z)-F(X(t),p_{x}(t)+y_{2},q_{x}+z))\nonumber\\
&\leq&-d_{1}(X(t))|y_{1}-y_{2}|^{2}.
\end{eqnarray}
and
\begin{eqnarray}
|g(t,y,z_{1})-g(t,y,z_{2})|&=&|F(X(t),p_{x}(t)+y,q_{x}+z_{1})-F(X(t),p_{x}(t)+y,q_{x}+z_{1})|\nonumber\\
&\leq& d_{2}|z_{1}-z_{2}|.
\end{eqnarray}
Moreover,
\begin{eqnarray}
\label{L2 BSDE g(x,o.0) condition}
&&E_{x}[\int_{0}^{\infty}e^{2\int_{0}^{t}d(X(u))du}|g(X(t),0,0)|^{2}dt]\nonumber\\
&\leq& E_{x}[\int_{0}^{\infty}e^{2\int_{0}^{t}d(X(u))du}|F(X(t),p_{x}(t),q_{x}(t))|^{2}dt]\nonumber\\
&\leq& E_{x}[\int_{0}^{\infty}e^{2\int_{0}^{t}d(X(u))du}|K(t)|^{2}dt]\nonumber\\
&<& \infty.
\end{eqnarray}
 $g$ satisfies all the conditions of the Lemma 2.2. Hence, there exist processes $(k_{x},l_{x})$ such that
$$
dk_{x}(t)=-g(t,k_{x}(t),l_{x}(t))dt+<l_{x}(t),dM_{x}(t)>,
$$
and
$$
e^{\int_{0}^{t}d(X(u))du}k_{x}(t)\rightarrow 0,
$$
as $t\rightarrow\infty$.\\
Putting $Y_{x}(t)=p_{x}(t)+k_{x}(t)$ and $Z_{x}(t)=q_{x}(t)+l_{x}(t)$,  we find that $(Y_{x}(t),Z_{x}(t))$ satisfies the following equation
$$
dY_{x}(t)=e^{\int_{0}^{t}\tilde{q}(X(u))du}\phi(X(t))dL_{t}-F(t,Y_{x}(t),Z_{x}(t))dt+<Z_{x}(t),dM_{x}>.
$$
and
$$
\lim_{t\rightarrow \infty}e^{\int_{0}^{t}d(X(u))du}Y_{t}=0 .
$$
\end{proof}
\begin{corollary}
\label{cor. y0 upper bounds}
Suppose  all the assumptions in Theorem 2.1 hold. If, in addition,
\begin{eqnarray}
\sup_{x}E_{x}[\int_{0}^{\infty}e^{\int_{0}^{t}d(X(u))du}|K(t)|^{2}dt]<\infty,\nonumber
\end{eqnarray}
then it follows that
\begin{eqnarray}
\sup_{x\in D}|Y_{x}(0)|<\infty.\nonumber
\end{eqnarray}
\end{corollary}
\begin{proof}\\
As shown in the proof of Theorem 2.1, $Y_{x}(t)$ has the  decomposition: $Y_{x}(t)=p_{x}(t)+k_{x}(t)$.\\
Setting $t=0$ in $(\ref{formula of px})$, it follows that
\begin{eqnarray}
|p_{x}(0)|&\leq& E_{X(t)}[|\int_{0}^{\infty}e^{\int_{0}^{l}\tilde{q}(X(u))du}\Phi(X(l))dL_{l}|]\nonumber\\
&\leq& \|\Phi\|_{\infty}\sup_{x}E_{x}[\int_{0}^{\infty}e^{\int_{0}^{l}\tilde{q}(X(u))du}dL_{l}]\nonumber\\
&<&\infty.
\end{eqnarray}
By Ito's formula, we obtain
\begin{eqnarray}
de^{2\int_{0}^{t}d(X(u))du}|k_{x}(t)|^{2}&=&-2e^{2\int_{0}^{t}d(X(u))du}k_{x}(t)g(t,k_{x}(t),l_{x}(t))dt\nonumber\\
&+&2e^{2\int_{0}^{t}d(X(u))du}k_{x}(t)d(X(t))dt+2e^{2\int_{0}^{t}d(X(u))du}k_{x}(t)<l_{x}(t),dM_{x}(t)>\nonumber\\
&+&e^{2\int_{0}^{t}d(X(u))du}\langle A(X(t))l_{x}(t),l_{x}(t)\rangle dt\nonumber
\end{eqnarray}
Choosing  two positive numbers $\delta_{1}$ and $\delta_{2}$ such that $\delta_{1}>\frac{1}{2\lambda}$ and
$\delta_{1}+\delta_{2}<\delta$, similar calculations as in the proof of Theorem 2.1 yield that, for any $t<T$,
\begin{eqnarray*}
&&E_{x}[e^{2\int_{0}^{t}d(X(u))du}|k_{x}(t)|^{2}]
+\frac{1}{\lambda}(1-\frac{1}{2\lambda\delta_{1}})E_{x}[\int_{t}^{T}e^{2\int_{0}^{s}d(X(u))du}|l_{x}(s)|^{2}ds]\nonumber\\
&\leq&E_{x}[e^{2\int_{t}^{T}d(X(u))du}|k_{x}(T)|^{2}]
+\frac{1}{2\delta_{2}d_{2}^{2}}E_{x}[\int_{t}^{T}e^{2\int_{0}^{s}d(X(u))du}|g(s,0,0)|^{2}ds].
\end{eqnarray*}
Setting $t=0$, we have
\begin{eqnarray}
|k_{x}(0)|^{2}= E_{x}[|k_{x}(0)|^{2}]&\leq&  E_{x}[e^{2\int_{0}^{T}d(X(u))du}|k_{x}(T)|^{2}]\nonumber\\
&+& \frac{1}{2\delta_{2}d_{2}^{2}}E_{x}[\int_{0}^{T}e^{2\int_{0}^{s}d(X(u))du}|g(s,0,0)|^{2}ds].\nonumber
\end{eqnarray}
Let  $T\rightarrow \infty$ to obtain that
\begin{eqnarray*}
\sup_{x}|k_{x}(0)|\leq \left(\frac{1}{2\delta_{2}d_{2}^{2}}\sup_{x}E_{x}[\int_{0}^{\infty}e^{2\int_{0}^{s}d(X(u))du}|g(s,0,0)|^{2}ds]\right)^{\frac{1}{2}}
<\infty,
\end{eqnarray*}
where the fact that $e^{2\int_{0}^{T}d(X(u))du}k_{x}(T)\rightarrow 0$ as $T\rightarrow \infty$, has been used.
Hence, we have  $\displaystyle{\sup_{x}|Y_{x}(0)|\leq \sup_{x}|p_{x}(0)|+\sup_{x}|k_{x}(0)|}<\infty$.
\end{proof}

\section{Linear PDEs}
\setcounter{equation}{0}
Set
$$
L_{2}=\frac{1}{2}\nabla\cdot(A\nabla ) +b\cdot \nabla +q
$$
where $b=(b_{1},...,b_{d})$  is a $R^{d}$-valued Borel measurable  function, and $q$ is a Borel measurable function on $R^{d}$ such that:
$$
I_{D}(|b|^{2}+|q|)\in L^{p}(D), \quad p>\frac{d}{2}.
$$
In this section, we solve the following linear boundary value problem:
\begin{eqnarray}
\label{simple}
\left\{\begin{array}{ll}
 \frac{1}{2}\nabla\cdot(A\nabla u)(x)+b\cdot \nabla u (x)+q(x)u(x)=F(x),&\textrm{on $D$}\\
\frac{1}{2}\frac{\partial u}{\partial{\gamma}}(x)=\phi &\textrm{on $\partial D$ },
\end{array}\right.
\end{eqnarray}
where $F$ and $\phi$ are bounded measurable functions on $D$.\\\\
It is well known that operator $L_{2}$ defined on a bounded domain D with Neumann boundary condition $\frac{\partial u}{\partial{\gamma}}(x)=0$ is associated with the quadratic form:
\begin{eqnarray}
\mathcal{E}(f,g):&=&-\int_{D}L_{2}f(x)g(x)dx \nonumber\\
&=&\frac{1}{2}\int_{D}\langle A\nabla f,\nabla g\rangle dx-\int_{D}b\cdot\nabla f (x)g(x)dx-\int_{D}q(x)f(x)g(x)dx\nonumber
\end{eqnarray}
\begin{definition} A bounded continuous function $u(x)$ defined on D is a weak solution of the problem $(\ref{simple})$ if $u\in W^{1,2}(D)$, and for any $g\in C^{\infty}(\overline{D})$,
\begin{eqnarray}
\mathcal{E}(u,g)=\int_{\partial D}\phi(x)g(x)\sigma(dx)-\int_{D}F(x)g(x)dx\nonumber,
\end{eqnarray}
where $\sigma$ denotes the $d-1$ dimensional Lebesgue measure on $\partial D$.
\end{definition}
Consider the operator
\begin{eqnarray}
L_{0}=\frac{1}{2}\nabla\cdot(A\nabla u)
\end{eqnarray}
on domain D with boundary condition $\frac{\partial u}{\partial {\gamma}}=0$ on $\partial D$. \\
$L_{0}$ is associated with a reflecting diffusion process $(X^{0},P^{0}_{x})$. By $\cite{LS}$, $X^{0}$ has the following decomposition:
\begin{eqnarray}
\label{decomposition of X0}
&&dX^{0}_{t}=\sigma(X^{0}_{t})dW_{t}+\frac{1}{2}\nabla A(X^{0}_{t})dt+\gamma (X^{0}_{t})dL^{0}_{t}, \nonumber\\
&&L^{0}_{t}=\int_{0}^{t}I_{\{X^{0}_{s}\in \partial D\}}dL^{0}_{s},
\end{eqnarray}
where the matrix $\sigma(x)$ is the positive definite symmetric square root of the matrix $A(x)$ and $\{W_{t}\}_{t>0}$ is a d-dimensional standard Brownian motion.\\
It is well known that operator $L_{0}$ is associated with the regular Dirichlet form:
$$
\mathcal{E}^{0}(u,v)=\frac{1}{2}\int_{D}a_{ij}\frac{\partial u}{\partial x_{i}}\frac{\partial v}{\partial x_{j}}dx
$$
and the domain of $\mathcal{E}^{0}$ is $W^{1,2}(D):=\{u\in L^{2}(D): \frac{\partial u}{\partial x_{i}}\in L^{2}(D) \}$.\\\\
The following lemma can be proved similarly as the Corollary 3.8 in $\cite{PEI}$ using the heat kernel estimates in $\cite{YZ}$.
\begin{lemma}
There exists a constant $K>0$, such that
$$
\sup_{x\in \bar{D}}E^{0}_{x}[L^{0}_{t}]\leq K\sqrt{t}\quad and \quad \inf_{x\in \bar{D}}E^{0}_{x}[L^{0}_{t}]>0.
$$
Moreover, we have $\sup_{x\in \bar{D}}E^{0}_{x}[(L^{0}_{t})^{n}]\leq K_{n}t^{\frac{n}{2}}$, for some constant $K_{n}>0$.
\end{lemma}
\vspace{5mm}
\noindent{Set} $M^{0}_{t}=\int_{0}^{t}\sigma(X^{0}_{s})dW_{s}$ and
\begin{eqnarray}
Z_{t}=e^{\int^{t}_{0}<A^{-1}b(X^{0}_{s}),dM^{0}_{s}>-\frac{1}{2}\int_{0}^{t}bA^{-1}b^{*}(X^{0}_{s})ds+\int_{0}^{t}q(X^{0}_{s})ds},
\end{eqnarray}
where $b^{*}$ is the transpose of the row vector $b$.\\

The proof of  the following two lemmas are inspired by that of the Lemma 2.1 and Theorem 2.2 in $\cite{PEI}$.
\begin{lemma}
For $t>0$, there are two strictly positive functions $M_{1}(t)$ and $M_{2}(t)$ such that, for any $x\in \overline{D}$, $M_{1}(t)\leq E^{0}_{x}[\int_{0}^{t}Z_{s}dL^{0}_{s}]\leq M_{2}(t)$. Furthermore, $M_{2}(t)\rightarrow 0$ as $t\rightarrow 0$.
\end{lemma}
\begin{proof}\\
$1^{\circ}$: Put
\begin{eqnarray}
\label{definition of tilde M}
&&\tilde{M}(t)=e^{\int^{t}_{0}<A^{-1}b(X^{0}_{s}),dM^{0}_{s}>-\frac{1}{2}\int_{0}^{t}bA^{-1}b^{*}(X^{0}_{s})ds},\\
&&e_{q}(t)=e^{\int_{0}^{t}q(X^{0}_{s})ds},\nonumber\\
&&M_{q}(t)=\sup_{x\in \overline{D}}E^{0}_{x}[\int_{0}^{t}|q(X^{0}_{s})|ds].\nonumber
\end{eqnarray}
Then we have
\begin{eqnarray}
\sup_{x\in \overline{D}}E^{0}_{x}[\int_{0}^{t}Z_{s}dL^{0}_{s}]&=&\sup_{x\in \overline{D}}E^{0}_{x}[\int_{0}^{t}\tilde{M}(s)e_{q}(s)dL^{0}_{s}]\nonumber\\
&\leq& \sup_{x\in \overline{D}}E^{0}_{x}[\max_{0\leq s\leq t}|\tilde{M}(s)|^{2}]^{\frac{1}{2}}\cdot \sup_{x\in \overline{D}}E^{0}_{x}[e_{2|q|}(t)(L^{0}_{t})^{2}]^{\frac{1}{2}}\nonumber\\
&\leq&\underbrace{\sup_{x\in \overline{D}}E^{0}_{x}[|\tilde{M}(t)|^{2}]^{\frac{1}{2}}}_{(I)}\cdot \underbrace{\sup_{x\in \overline{D}}E^{0}_{x}[e_{4|q|}(t)]^{\frac{1}{4}}}_{(II)}\cdot \underbrace{\sup_{x\in \overline{D}}E^{0}_{x}[(L^{0}_{t})^{4}]^{\frac{1}{4}}}_{(III)}
\end{eqnarray}
By Khash'Minskii's lemma and Theorem 2.1 in $\cite{LLZ}$, $(I)$ and $(II)$ are bounded if t belongs to a bounded interval.
Because of $E^{0}_{x}[(L^{0}_{t})^{n}]\leq K_{n} t^{\frac{n}{2}}$, we see that $M_{2}(t):=K(I)(II)\sqrt{t}$ is the required upper bound.\\
$2^{\circ}$: Since
\begin{eqnarray}
E^{0}_{x}[L^{0}_{t}]^{2}\leq E^{0}_{x}[\int^{t}_{0}\tilde{M}^{-1}(s)e_{-q}(s)dL^{0}_{s}]\cdot E^{0}_{x}[\int^{t}_{0}\tilde{M}(s)e_{q}(s)dL^{0}_{s}],
\end{eqnarray}
we obtain
\begin{eqnarray}
E^{0}_{x}[\int^{t}_{0}\tilde{M}(s)e_{q}(s)dL^{0}_{s}]\geq \frac{E^{0}_{x}[L^{0}_{t}]^{2}}{E^{0}_{x}[\int^{t}_{0}\tilde{M}^{-1}(s)e_{-q}(s)dL^{0}_{s}]}.
\end{eqnarray}
Here
\begin{eqnarray}
\tilde{M}^{-1}(t)&=&e^{-\int_{0}^{t}<A^{-1}b(X^{0}_{s}),dM^{0}_{s}>
+\frac{1}{2}\int_{0}^{t}bA^{-1}b^{*}(X^{0}_{s})ds}\nonumber\\
&=& e^{-\int_{0}^{t}<A^{-1}b(X_{s}),dM^{0}_{s}>-\frac{1}{2}\int_{0}^{t}bA^{-1}b^{*}(X^{0}_{s})ds}\cdot e^{\int_{0}^{t}bA^{-1}b^{*}(X^{0}_{s})ds}\nonumber\\
&:=& N(t)\cdot e^{\int_{0}^{t}bA^{-1}b^{*}(X^{0}_{s})ds}
\end{eqnarray}
By the proof of the first part, replacing $\tilde{M}_{t}$, $q$ by $N_{t}$ and $bA^{-1}b^{*}-q$ respectively, it is seen that there exists $K(t)>0$ such that $\sup_{x\in \overline{D}}E^{0}_{x}[\int^{t}_{0}\tilde{M}^{-1}(s)e_{-q}(s)dL^{0}_{s}]\leq K(t)$.\\
As $\inf_{x\in \overline{D}}E^{0}_{x}[L^{0}_{t}]>0$,  we complete the proof of the lemma by setting $M_{1}(t)=\frac{\inf_{x\in \overline{D}}E^{0}_{x}[L^{0}_{t}]^{2}}{K(t)}$.$\square$
\end{proof}
\\\\
Set $G(x):= E^{0}_{x}[\int_{0}^{\infty}Z_{s}dL^{0}_{s}]$.
\begin{lemma}
\label{upper boundary of semigroup}
If there is a point $x_{0}\in \overline{D}$, such that $G(x_{0})<\infty$, then there are two positive constants K and $\beta$ such that $\sup_{x\in \overline{D}}E^{0}_{x}[Z_{t}]\leq Ke^{-\beta t}$.
\end{lemma}
\begin{proof}\\
 By Girsanov Theorem and Feymann-Kac formula, $L_{2}= \frac{1}{2}\nabla\cdot(A\nabla) +b\cdot\nabla+q$ is associated with the semigroup $\{T_{t}\}_{t>0}$, where $T_{t}f(x)=E^{0}_{x}[Z_{t}f(X^{0}_{t})]$ for $f\in L^{2}(D)$.\\
By the upper and lower bound estimates of the heat kernel $p_{2}(t,x,y)$ associated with $T_{t}$ in $\cite{YZ}$,
the following inequality holds,
\begin{eqnarray}
c^{-1}\int_{D}f(x)dx\leq E^{0}_{x}[Z_{1}f(X^{0}_{1})]\leq c\int_{D}f(x)dx,
\end{eqnarray}
where $c$ is a positive constant.
Since
\begin{eqnarray}
G(x)=\sum_{n=0}^{\infty}E^{0}_{x}[Z_{n}E^{0}_{X^{0}_{n}}[\int_{0}^{1}Z_{s}L^{0}(ds)]]\geq M_{1}(1)\sum_{n=0}^{\infty}E^{0}_{x}[Z_{n}]\nonumber
\end{eqnarray}
and $G(x_{0})<\infty$, there is a positive integer number $N$ such that
$$
\frac{1}{2c^{2}}\geq E^{0}_{x_{0}}[Z_{N}]= E^{0}_{x_{0}}[Z_{1}E^{0}_{X_{1}}[Z_{N-1}]]
\geq c^{-1}\int_{D}E^{0}_{x}[Z_{N-1}]m(dx).
$$
This implies
$$
\int_{D}E^{0}_{x}[Z_{N-1}]m(dx)\leq \frac{1}{2c}.
$$
Thus
\begin{eqnarray}
\label{upper boundes of Zn}
\sup_{x\in \overline{D}}E^{0}_{x}[Z_{N}]=\sup_{x\in \overline{D}}E^{0}_{x}[Z_{1}E^{0}_{X_{1}}[Z_{N-1}]]\leq c\int_{D}E^{0}_{x}[Z_{N-1}]m(dx)\leq \frac{1}{2}.
\end{eqnarray}
For any $t>0$, there exists a positive number $n$ such that $\frac{t}{N}\in[n-1,n)$. Then by $(\ref{upper boundes of Zn})$, it follows that
\begin{eqnarray}
E^{0}_{x}[Z_{t}]\leq \frac{1}{2^{n-1}}E^{0}_{x}[Z_{t-N(n-1)}]&\leq& \left(\sup_{x\in D,0\leq t\leq N}E^{0}_{x}[Z_{t}]\right)\frac{1}{2^{n-1}}\nonumber\\
&\leq& 2\sup_{x\in D,0\leq t\leq N}E^{0}_{x}[Z_{t}]e^{-\frac{ln2}{N}t}.\square
\end{eqnarray}
\end{proof}

\begin{theorem}
If there exists $x_{0}\in \overline{D}$ such that $G(x_{0})<\infty$, then there exists a unique bounded continuous weak solution of the problem $(\ref{simple})$:
\end{theorem}
\begin{proof}\\
$Existence:$\\
Due to Theorem 3.2 in $\cite{CZ2}$, there exists a unique, bounded, continuous weak solution $u_{2}$ of the following problem:
\begin{eqnarray}
\label{linear equation with F=0}
\left\{\begin{array}{ll}
 L_{2}u_{2}(x)=0,&\textrm{on $D$}\\
\frac{1}{2}\frac{\partial u_{2}}{\partial{\gamma}}(x)=\phi &\textrm{on $\partial D$ }.
\end{array}\right.
\end{eqnarray}
Thus by the linearity of the problem $(\ref{simple})$,  we only need to show that the following problem has a bounded continuous weak solution:
\begin{eqnarray}
\label{linear equation 2}
\left\{\begin{array}{ll}
 L_{2}u_{1}(x)=F(x),&\textrm{on $D$}\\
\frac{\partial u_{1}}{\partial{\gamma}}(x)=0 &\textrm{on $\partial D$ }
\end{array}\right.
\end{eqnarray}
The semigroup associated with operator $L_{2}$ is $\{T_{t},t>0\}$. By Lemma $\ref{upper boundary of semigroup}$, we have
$$
\sup_{x\in D}|T_{t}F(x)|=\sup_{x\in D}|E_{x}^{0}[Z_{t}F(X_{t}^{0})]|\leq Ke^{-\beta t}\|F\|_{\infty}.
$$
Then
$$
u_{1}(x):=\int_{0}^{\infty}T_{t}F(x)dt
$$
is well defined and has the following bound:
$$
\sup_{x\in D}|u_{1}(x)|\leq \frac{K}{\beta}\|F\|_{\infty}.
$$
The function $u_{1}(x)$ is also continuous on D. \\
In fact, fixing any $x\in D$ and $\epsilon>0$, we can firstly choose a constant $t_{0}>0$, such that $\sup_{z\in D}|\int_{0}^{t_{0}}T_{s}F(z)ds|<\frac{\epsilon}{3}$. And because $T_{t_{0}}u_{1}(x)$ is continuous, there exists a constant $\delta>0$, such that for any $y$ with $|y-x|<\delta$, $|T_{t_{0}}u_{1}(x)-T_{t_{0}}u_{1}(y)|\leq \frac{\epsilon}{3}$.\\
We find that
\begin{eqnarray}
T_{t}u_{1}(x)=E_{x}^{0}[Z_{t}u_{1}(X_{t}^{0})]&=&E_{x}^{0}[Z_{t}\int_{0}^{\infty}E_{X^{0}_{t}}[Z_{s}F(X^{0}_{s})]ds]\nonumber\\
&=&\int_{0}^{t}E_{x}^{0}[Z_{t+s}u_{1}(X_{t+s}^{0})]ds\nonumber\\
&=&\int_{t}^{\infty}T_{s}F(x)ds\nonumber\\
&=&u_{1}(x)-\int_{0}^{t}T_{s}F(x)ds.
\end{eqnarray}
For any $y$ satisfying $|y-x|<\delta$, it follows that
\begin{eqnarray}
|u_{1}(x)-u_{1}(y)|\leq |T_{t_{0}}u_{1}(x)-T_{t_{0}}u_{1}(y)|+|\int_{0}^{t_{0}}T_{s}F(x)ds|+|\int_{0}^{t_{0}}T_{s}F(y)ds|\leq \epsilon.
\end{eqnarray}
This implies that the function $u_{1}$ is continuous on domain $D$.\\
Denote the resolvents associated with operator $L_{2}$ by $\{G_{\beta},\beta>0\}$.
Note that
\begin{eqnarray}
G_{\beta}u_{1}(x)&=&\int_{0}^{\infty}e^{-\beta t}T_{t}u_{1}(x)dt\nonumber\\
&=&\int_{0}^{\infty}e^{-\beta t}u_{1}(x)dt-\int_{0}^{\infty}e^{-\beta t}\int_{0}^{t}T_{s}F(x)dsdt\nonumber\\
&=&\frac{1}{\beta}u_{1}(x)-\int_{0}^{\infty}\int_{0}^{t}e^{-\beta t}T_{s}F(x)dsdt\nonumber\\
&=&\frac{1}{\beta}u_{1}(x)-\int_{0}^{\infty}T_{s}F(x)(\int_{s}^{\infty}e^{-\beta t}dt)ds\nonumber\\
&=&\frac{1}{\beta}u_{1}(x)-\frac{1}{\beta}G_{\beta}F(x).
\end{eqnarray}
We have
$$
\beta(u_{1}(x)-\beta G_{\beta}u_{1}(x))=\beta G_{\beta}F(x).
$$
Therefore,
$$
\lim_{\beta\rightarrow\infty}\int_{D}\beta(u_{1}(x)-\beta G_{\beta}u_{1}(x))u_{1}(x)dx=\lim_{\beta\rightarrow\infty}\int_{D}\beta G_{\beta}F(x)u(x)dx=\int_{D}F(x)u_{1}(x)dx<\infty.
$$
This implies that $u_{1}\in D(\mathcal{E})$ (see $\cite{MR}$) and $u_{1}$ is a weak solution of equation $(\ref{linear equation 2})$.
By the linearity,  $u=u_{1}+u_{2}$ is a bounded continuous weak solution of equation
$(\ref{simple})$.\\
$Uniqueness:$\\
Let $v_{1}$ and $v_{2}$ be two bounded continuous weak solutions of the equation $(\ref{simple})$. Then $v_{1}-v_{2}$ is the solution of equation $(\ref{linear equation with F=0})$ with $\phi=0$. Then by the uniqueness of the equation $(\ref{linear equation with F=0})$ proved in $\cite{CZ2}$, we know that $v_{1}=v_{2}$. $\square$
\end{proof}

\section{Semilinear PDEs}
\setcounter{equation}{0}
Recall that
\begin{eqnarray}
L_{1}=\frac{1}{2}\sum_{i,j=1}^d \frac{\partial}{\partial x_i}\left(a_{ij}(x)\frac{\partial}{\partial x_j}\right)+\sum_{i=1}^d  b_i(x)\frac{\partial}{\partial x_i}\nonumber
\end{eqnarray}
and $L_{2}=L_{1}+q$ are two operators both defined on the domain D and equipped with the Neumann boundary condition $\frac{\partial}{\partial \gamma}=0$ on $\partial  D$.\\
$(\Omega, \mathfrak{F}_{t}, X(t), P_{x}, x\in D)$ is the reflecting diffusion process associated with the operator $L_{1}$ with the decomposition introduced in $(\ref{deomposition of X})$.\\\\
In this section, we solve the following semilinear boundary value problem:
\begin{eqnarray}
\label{semilinear L2}
\left\{\begin{array}{ll}
{L_{2}}u(x)=-G(x,u(x),\nabla u(x)),&\textrm{on $D$}\\
\frac{1}{2}\frac{\partial u}{\partial{\gamma}}(x)=\phi(x) &\textrm{on $\partial D$ }
\end{array}\right.
\end{eqnarray}
Let $\mathcal{E}(\cdot,\cdot)$ be the quadratic form associated with the operator $L_{2}$:
$$
\mathcal{E}(u,v)=\frac{1}{2}\int_{D}<A\nabla u,\nabla v >dx-\int_{D}<b,\nabla u>vdx-\int_{D}quvdx.
$$
\begin{definition} A bounded continuous function $u(x)$ defined on D is called a weak solution of the equation
$(\ref{semilinear L2})$
if $u\in W^{1,2}(D)$, and for any $g\in C^{\infty}(\overline{D})$,
\begin{eqnarray}
\mathcal{E}(u,g)=\int_{\partial D}\phi(x)g(x)\sigma(dx)+\int_{D}G(x,u(x),\nabla u(x))g(x)dx\nonumber.
\end{eqnarray}
\end{definition}
Recall that $L_{t}$ is the boundary local time of $X(t)$ defined in $(\ref{deomposition of X})$ and $L^{0}_{t}$ is the boundary local time of $X^{0}_{t}$ in $(\ref{decomposition of X0})$.\\
As a consequence of the Girsanov theorem, we  have:
\begin{lemma}
\label{relationship between the integrale w.r.t. local times}
Suppose that the function $f$ satisfies $E_{x}[\int_{0}^{T}e^{\int_{0}^{t}f(X(u))du}dL_{t}]<\infty$. Then it holds that
 \begin{eqnarray}
E_{x}[\int_{0}^{T}e^{\int_{0}^{t}f(X(u))du}dL_{t}]
=E^{0}_{x}[\int_{0}^{T}\tilde{M}_{t}e^{\int_{0}^{t}f(X^{0}_{u})du}dL^{0}_{t}],\nonumber
\end{eqnarray}
where $\tilde{M}_{t}$ was defined in $(\ref{definition of tilde M})$.
\end{lemma}

The following lemma is deduced from Theorem 3.2 in $\cite{CZ2}$.
\begin{lemma}
\label{boundedness of gauge function}
Suppose that the function $\tilde{q}\in L^{p}(D)$ and $p>\frac{d}{2}$. If there exists some point $x_{0}\in D$, such that
\begin{eqnarray}
\label{integrability w.r.t local time at x0}
E_{x_{0}}[\int_{0}^{\infty}e^{\int_{0}^{t}\tilde{q}(X(u))du}dL_{t}]<\infty,
\end{eqnarray}
then it holds that
\begin{eqnarray}
\sup_{x}E_{x}[\int_{0}^{\infty}e^{\int_{0}^{t}\tilde{q}(X(u))du}dL_{t}]<\infty.\nonumber
\end{eqnarray}
\\
\end{lemma}

Let $G(x,y,z):R^{d}\times R\times R^{d}\rightarrow R$ be a bounded Borel measurable function. Introduce
the following conditions:\\
\textbf{(H.1)} $(y_{1}-y_{2})(G(x,y_{1},z)-G(x,y_{2},z))\leq -h_{1}(x)|y_{1}-y_{2}|^{2}$,\\
\textbf{(H.2)} $|G(x,y,z_{1})-G(x,y,z_{2})|\leq h_{2}|z_{1}-z_{2}|$.\\
Set $h(t)=-h_{1}(X(t))+\delta h^{2}_{2}+q(X(t))$ and $\tilde{h}(t)=-h_{1}(X(t))+\delta h^{2}_{2}$ for some constant $\delta>\frac{1}{2\lambda}$.\\
\begin{theorem}
\label{semilinear equation L2}
Suppose that the conditions (H.1) and (H.2) are satisfied. Assume
\begin{eqnarray}
\label{condition1 in main theorem}
E_{x_{1}}[\int_{0}^{\infty}e^{2\int_{0}^{t}(q(X(u))+\tilde{h}(u))du}dt]<\infty,\quad for \quad some \quad x_{1}\in D,
\end{eqnarray}
and there exists some point $x_{0}\in D$, such that
\begin{eqnarray}
\label{condition2 in main theorem}
E_{x_{0}}[\int_{0}^{\infty}e^{\int_{0}^{t}q(X(u))du}dL_{t}]<\infty.
\end{eqnarray}
Then the semilinear Neumann boundary value problem $(\ref{semilinear L2})$ has a unique continuous weak solution.
\end{theorem}
\begin{proof}\\
Set
$$\tilde{G}(X(t),y,z):=e^{\int_{0}^{t}q(X(u))dt}G(x,e^{-\int_{0}^{t}q(X(u))dt}y,e^{-\int_{0}^{t}q(X(u))dt}z).$$
Then
\begin{eqnarray}
(y_{1}-y_{2})(\tilde{G}(X(t),y_{1},z)-\tilde{G}(X(t),y_{2},z))\leq -h_{1}(x)|y_{1}-y_{2}|^{2}
\end{eqnarray}
and
\begin{eqnarray}
&&|\tilde{G}(X(t),y,z_{1})-\tilde{G}(X(t),y,z_{2})|\leq h_{2}|z_{1}-z_{2}|.
\end{eqnarray}
Note that
$$
\tilde{G}(X(t),y,z)\leq e^{\int_{0}^{t}q(X(u))dt}\|G\|_{\infty}.
$$
By Theorem 2.1 there exists a unique process $(\hat{Y}_{x},\hat{Z}_{x})$ satisfying
\begin{eqnarray}
&&d\hat{Y}_{x}(t)=-\tilde{G}(X(t),\hat{Y}_{x}(t),\hat{Z}_{x}(t))dt+e^{\int_{0}^{t}q(X(u))du}\phi(X(t))dL(t)+\langle \hat{Z}_{x}(t),dM_{x}(t)\rangle\nonumber\\
&&e^{\int_{0}^{t}\tilde{h}(u)du}\hat{Y}_{x}(t)\rightarrow 0 \quad  as \quad t\rightarrow\infty.\nonumber
\end{eqnarray}
Furthermore, Corollary $\ref{cor. y0 upper bounds}$ implies that $\displaystyle{\sup_{x}\hat{Y}_{x}(0)}<\infty$.\\
From Ito's formula, it follows that
\begin{eqnarray}
&&d(e^{-\int_{0}^{t}q(X(u))dt}\hat{Y}_{x}(t))\nonumber\\
&=&-q(X(t))e^{-\int_{0}^{t}q(X(u))dt}\hat{Y}_{x}(t)dt-e^{-\int_{0}^{t}q(X(u))dt}\tilde{G}(X(t),\hat{Y}_{x}(t),\hat{Z}_{x}(t))dt\nonumber\\
&&+\phi(X(t))dL_{t}+<e^{-\int_{0}^{t}q(X(u))dt}\hat{Z}_{x}(t),dM_{x}(t)>.\nonumber
\end{eqnarray}
Setting $Y_{x}(t):=e^{-\int_{0}^{t}q(X(u))dt}\hat{Y}_{x}(t)$ and $Z_{x}(t):=e^{-\int_{0}^{t}q(X(u))dt}\hat{Z}_{x}(t)$, we obtain
$$
dY_{x}(t)=-(q(X(t))Y_{x}(t)+{G}(X(t),Y_{x}(t),Z_{x}(t)))dt+\phi(X(t))dL_{t}+<{Z}_{x}(t),dM_{x}(t)>.
$$
Moreover,
\begin{eqnarray}
e^{\int_{0}^{t}h(u)dt}Y_{x}(t)=e^{\int_{0}^{t}h(u)dt}e^{-\int_{0}^{t}q(X(u))dt}\hat{Y}_{x}(t)=
e^{\int_{0}^{t}\tilde{h}(X(u))dt}\hat{Y}_{x}(t)\rightarrow0\quad as\quad t\rightarrow\infty.
\end{eqnarray}
So by Ito's formula, we have that, for any $t<T$,
\begin{eqnarray}
\label{first BSDE}
&&e^{\int_{0}^{t}h(u)du}Y_{x}(t)\nonumber\\
&=&e^{\int_{0}^{T}h(u)du}Y_{x}(T)
+\int_{t}^{T}e^{\int_{0}^{s}h(u)du}\left(G(X_{x}(s),Y_{x}(s),Z_{x}(t))+q(X_{x}(s))Y_{x}(s)\right)ds\nonumber\\
&-&\int_{t}^{T}e^{\int_{0}^{s}h(u)du}\phi(X(s))dL_{s}
-\int_{t}^{T}h(s)e^{\int_{0}^{s}h(u)du}Y_{x}(s)ds\nonumber\\
&-&\int_{t}^{T}e^{\int_{0}^{s}h(u)du}\langle Z_{x}(t),dM_{x}(t)\rangle.
\end{eqnarray}
Put $u_{0}(x)=Y_{x}(0)$ and $v_{0}(x)=Z_{x}(0)$.\\
Since $Y_{x}(0)=\hat{Y}_{x}(0)$, we know that $u_{0}$ is a bounded function on domain $D$.
By the Markov property of $X$ and the uniqueness of $(Y_{x},Z_{x})$ , it is easy to see that
$$
Y_{x}(t)=u_{0}(X(t)),\quad Z_{x}(t)=v_{0}(X(t)).
$$
So that $\displaystyle{\sup_{x\in D, t>0}|Y_{x}(t)|}\leq \|u_{0}\|_{\infty}<\infty.$\\
Now consider the following problem:
\begin{eqnarray}
\label{linear case}
\left\{\begin{array}{ll}
{L_{2}}u(x)=-G(x,u_{0}(x),v_{0}(x)),&\textrm{on $D$}\\
\frac{1}{2}\frac{\partial u}{\partial{\gamma}}(x)=\phi(x) &\textrm{on $\partial D$ }
\end{array}\right.
\end{eqnarray}
By Theorem 3.1, problem $(\ref{linear case})$ has a unique continuous weak solution $u(x)$. Next we will show that $u=u_{0}$. \\
Since $u$ belongs to the domain of the Dirichlet form associated with the process $X(t)$, it follows from the Fukushima's decomposition that:
\begin{eqnarray}
&&du(X(t))\nonumber\\
&=&-[G(X(t),u_{0}(X(t)),v_{0}(X(t)))+q(X(t))u(X(t))]dt+\phi(X(t))dL(t)+\langle\nabla u(X(t)),dM_{x}(t)\rangle\nonumber\\
&=&-[G(X(t),Y_{x}(t),Z_{x}(t))+q(X(t))u(X(t))]+\phi(X(t))dL(t)+\langle\nabla u(X(t)),dM_{x}(t)\rangle\nonumber
\end{eqnarray}
From the condition $(\ref{condition1 in main theorem})$ and the boundedness of $u(x)$,  it follows that
$$
\lim_{t\rightarrow\infty}E_{x}[e^{2\int_{0}^{t}h(u)du}u^{2}(X(t))]\leq \|u\|^{2}_{\infty}\lim_{t\rightarrow\infty}E_{x}[e^{2\int_{0}^{t}(\tilde{h}+q)(u)du}]=0.
$$
By Ito's formula, it follows that, for any $t<T$,
\begin{eqnarray}
\label{second BSDE}
&&e^{\int_{0}^{t}h(u)du}u(X(t))\nonumber\\
&=&e^{\int_{0}^{T}h(u)du}u(X(T))+\int_{t}^{T}e^{\int_{0}^{s}h(u)du}[G(X(s),Y_{x}(s),Z_{x}(s))+q(X(s))u(X(s))]ds\nonumber\\
&-&\int_{t}^{T}e^{\int_{0}^{s}h(u)du}\phi(X(s))dL(s)
-\int_{t}^{T}h(s)e^{\int_{0}^{s}h(u)du}u(X(s))ds\nonumber\\
&-&\int_{t}^{T}e^{\int_{0}^{s}h(u)du}\langle \nabla u(X(t)),dM_{x}(t)\rangle.
\end{eqnarray}
Set
$$
v_{x}(t)=u(X(t))-Y_{x}(t)\quad and\quad R_{x}(t)=\nabla u(X(t))-Z_{x}(t).
$$
Subtracting the equations $(\ref{first BSDE})$ from $(\ref{second BSDE})$, we obtain the following equation: for any $t<T$,
\begin{eqnarray}
\label{equation of v}
&&e^{\int_{0}^{t}h(u)du}v(X(t))\nonumber\\
&=&e^{\int_{0}^{T}h(u)du}v(X(T))+\int_{t}^{\infty}(q(X(u))-h(u))e^{\int_{0}^{s}h(u)du}v(X(s))ds\nonumber\\
&&-\int_{t}^{\infty}e^{\int_{0}^{s}h(u)du}<R_{x}(t), ,dM_{x}(t)>\nonumber\\
&=&e^{\int_{0}^{T}h(u)du}v(X(T))-\int_{t}^{T}\tilde{h}(s)e^{\int_{0}^{s}h(u)du}v(X(s))ds\nonumber\\
&&-\int_{t}^{T}e^{\int_{0}^{s}h(u)du}<R_{x}(t),dM_{x}(t)>.\nonumber
\end{eqnarray}
Set $g(t)=e^{\int_{0}^{t}h(u)du}v(t)$. Taking conditional expectation on both sides of $(\ref{equation of v})$,
we find that
\begin{eqnarray}
g(t)&=&E_{x}[g(T)-\int_{t}^{T}\tilde{h}(s)g(s)ds|\mathcal{F}_{t}]\nonumber\\
&=&E_{x}[g(T)(1-\int_{t}^{T}\tilde{h}(s)ds)+\int_{t}^{T}\int_{s}^{T}\tilde{h}(s)\tilde{h}(s_{1})g(s_{1})ds_{1}ds|\mathcal{F}_{t}]\nonumber\\
&=&E_{x}[g(T)(1-\int_{t}^{T}\tilde{h}(s)ds+\frac{1}{2}(\int_{t}^{T}\tilde{h}(s)ds)^{2})\nonumber\\
&&+(-1)^{3}\int_{t}^{T}\int_{s}^{T}\int_{s_{1}}^{T}\tilde{h}(s)\tilde{h}(s_{1})\tilde{h}(s_{2})g(s_{2})ds_{2}ds_{1}ds|\mathcal{F}_{t}].\nonumber
\end{eqnarray}
Keeping iterating, we obtain
\begin{eqnarray}
g(t)&=&E_{x}[g(T)(\sum_{k=0}^{n}\frac{(-\int_{t}^{T}\tilde{h}(s)ds)^{n}}{n!})\nonumber\\
&&+(-1)^{n+1}\int_{t}^{T}\int_{s}^{T}\int_{s_{1}}^{T}...\int_{s_{n-1}}^{T}
\tilde{h}(s)\tilde{h}(s_{1})...\tilde{h}(s_{n})g(s_{n})ds_{n}...ds_{1}ds|\mathcal{F}_{t}]\nonumber
\end{eqnarray}
Since $E_{x}[|g(T)|e^{\int_{t}^{T}|\tilde{h}|(s)ds}]<\infty$, letting $n\rightarrow \infty$, by dominated convergence theorem, it follows that
\begin{eqnarray}
g(t)=E_{x}[g(T)e^{-\int_{t}^{T}\tilde{h}(s)ds}|\mathcal{F}_{t}].\nonumber
\end{eqnarray}
Then
\begin{eqnarray}
v(t)=E_{x}[v(T)e^{\int_{t}^{T}(h(s)-\tilde{h}(s))ds}|\mathcal{F}_{t}]\leq (\|u_{0}\|_{\infty}+\|u\|_{\infty})E_{x}[e^{\int_{t}^{T}q(X(s))ds}|\mathcal{F}_{t}].
\end{eqnarray}
Hence, it follows that
\begin{eqnarray}
0\leq e^{\int_{0}^{t}q(X(s))ds}|v(t)|\leq (\|u_{0}\|_{\infty}+\|u\|_{\infty})\lim_{T\rightarrow\infty}E_{x}[e^{\int_{0}^{T}q(X(s))ds}|\mathcal{F}_{t}].
\end{eqnarray}
Since the condition $(\ref{condition2 in main theorem})$ implies
$$
\lim_{T\rightarrow\infty}E_{x}[e^{\int_{0}^{T}q(X(s))ds}]=0,
$$
we deduce that $E_{x}[e^{\int_{0}^{t}q(X(s))ds}|v(t)|]=0$ and hence  $v(t)=0$, $P_{x}-a.s.$.\\
Therefore, for any $t>0$, we have $u(X(t))=Y_{x}(t)$ and $\nabla u(X(t))=Z_{x}(t)$ by the uniqueness of the Doob-Meyer decomposition of semimartingales. In particular,  $u(x)=E_{x}[u(X_{x}(0))]=E_{x}[Y_{x}(0)]=u_{0}(x)$. This shows that $u(x)$ is a weak solution of the equation $(\ref{semilinear L2})$.\\
If $\tilde{u}$ is another solution of the problem $(\ref{semilinear L2})$. Then the processes $\tilde{Y}_{x}(t):=\tilde {u}(X(t))$ and $\tilde{Z}_{x}(t):=\nabla \tilde{u}(X(t))$ satisfy the following equation
\begin{eqnarray}
d\tilde{Y}_{x}(t)=-G(X(t),\tilde{Y}_{x}(t),\tilde{Z}_{x}(t))dt-\phi(X(t))dL_{t}+<\tilde{Z}_{x}(t),dM_{x}(t)>.
\end{eqnarray}
Set $\bar{Y}_{x}(t)=e^{\int_{0}^{t}q(X(u))du}\tilde{Y}_{x}(t)$ and $\bar{Z}_{x}(t)=e^{\int_{0}^{t}q(X(u))du}\tilde{Z}_{x}(t)$.\\
By chain rule, it follows that
$$
d\bar{Y}_{x}(t)=-\tilde{G}(X(t),\bar{Y}_{x}(t),\bar{Z}_{x}(t))dt+e^{\int_{0}^{t}q(X(u))du}\phi(X(t))dL(t)+\langle \bar{Z}_{x}(t),dM_{x}(t)\rangle
$$
Moreover, because $\tilde{u}$ is bounded, we have
$$
\lim_{t\rightarrow \infty}e^{\int_{0}^{t}\tilde{h}(u)du}\bar{Y}_{x}(t)=\lim_{t\rightarrow \infty}e^{\int_{0}^{t}{h}(u)du}\tilde {u}(X(t))=0.
$$
Therefore, from the uniqueness of the solution of the BSDE in Theorem 2.1, we have
$$
\tilde{Y}_{x}(t)={Y}_{x}(t)\quad \tilde{Z}_{x}(t)={Z}_{x}(t).
$$
In particular,
$$
\tilde{u}(x)=E_{x}[\tilde{Y}_{x}(t)]=E_{x}[{Y}_{x}(t)]=u(x).
$$
\end{proof}


\section{Semilinear Elliptic PDEs with Singular Coefficients}
\setcounter{equation}{0}
Recall the operator
$$
L=\frac{1}{2}\nabla\cdot(A\nabla)+B\cdot\nabla-\nabla\cdot(\hat{B}\cdot)+Q
$$
on the domain $D$ equipped with the mixed boundary condition on $\partial D$:
\begin{eqnarray}
\frac{1}{2}\frac{\partial u}{\partial \gamma}-\langle\hat{B},{n}\rangle u(x)=0.\nonumber
\end{eqnarray}
The quadratic form associated with $L$ is given by:
\begin{eqnarray}
\mathcal{Q}(u,v):=(-Lu,v)=&&\frac{1}{2}\sum_{i,j}\int_{D}a_{ij}(x)\frac{\partial u}{\partial x_{i}}\frac{\partial v}{\partial x_{j}}dx-\sum_{i}\int_{D}B_{i}(x)\frac{\partial u}{\partial x_{i}}v(x)dx\nonumber\\
&-&\sum_{i}\int_{D}\hat{B}_{i}(x)\frac{\partial v}{\partial x_{i}}u(x)dx-\int_{D}Q(x)u(x)v(x)dx,\nonumber
\end{eqnarray}
where $(\cdot.\cdot)$ stands for the inner product in $L^{2}(D)$.\\
The domain of the quadratic form is
$$
\mathcal{D}(\mathcal{Q})=W^{1,2}(D):=\{u:u\in L^{2}(D),\frac{\partial u}{\partial x_{i}}\in L^{2}(D),i=1,...,d\}.
$$
Let $\{S_{t}$, $t\geq 0\}$ denote the semigroup generated by $L$.\\\\
In this section, our  aim is to solve the following equation:
\begin{eqnarray}
\label{final equation}
\left\{\begin{array}{ll}
{L}f(x)=-F(x,f(x)),&\textrm{on $D$}\\
\frac{1}{2}\frac{\partial f}{\partial \gamma}(x)-<\widehat{B},n>(x)f(x)=\Phi(x) &\textrm{on $\partial D$ }
\end{array}\right.
\end{eqnarray}
\begin{definition}
A bounded continuous function $f(x)$ defined on D is called a weak solution of the equation $(\ref{final equation})$ if $f\in W^{1,2}$, and for any $g\in C^{\infty}(\bar{D})$,
\begin{eqnarray}
\mathcal{Q}(u,g)=\int_{\partial D}\Phi(x)g(x)\sigma(dx)+\int_{D}F(x,u(x))g(x)dx\nonumber.
\end{eqnarray}
\end{definition}
Here the function $F:R^{d}\times R\rightarrow R$ is a bounded measurable function and satisfies the following condition:\\
\textbf{(E.1)} $ (y_{1}-y_{2})(F(x,y_{1})-F(x,y_{2}))\leq-r_{1}(x)|y_{1}-y_{2}|^{2}$.\\
\\
Recall the following regular Dirichlet form
\begin{eqnarray}
\left\{\begin{array}{ll}
\mathcal{E}^{0}(u,v)=\frac{1}{2}\sum_{i,j}\int_{D}a_{ij}(x)\frac{\partial u}{\partial x_{i}}\frac{\partial v}{\partial x_{j}}dx,\\
D(\mathcal{E}^{0})=W^{1,2}(D)
\end{array}\right.
\end{eqnarray}
associated with the operator $L_{0}=\frac{1}{2}\nabla(A\nabla)$ equipped with the Neumann boundary condition $\frac{\partial}{\partial \gamma}=0$ on $\partial D$.\\
The associated reflecting diffusion process is denoted by $\{\Omega,\mathfrak{F}_{t},X^{0}_{t},\theta^{0}_{t},\gamma^{0}_{t},P^{0}_{x}\}$. Here $\theta^{0}_{t}$ and $\gamma^{0}_{t}$ are the shift and reverse operators defined by
\begin{eqnarray}
X^{0}_{s}(\theta^{0}_{t}(\omega))&=&X^{0}_{t+s}(\omega),  s,t\geq0 \nonumber\\
X^{0}_{s}(\gamma^{0}_{t}(\omega))&=&X^{0}_{t-s}(\omega),   s\leq t \nonumber.
\end{eqnarray}
The process $(X^{0}_{t})_{t\geq0}$ has the decomposition in $(\ref{decomposition of X0})$.
The martingale part of $X^{0}_{t}$ is $M^{0}_{t}=\int_{0}^{t}\sigma(X^{0}_{s})dW_{s}$.\\
The following probabilistic representation of semigroup $S_{t}$ was proved in ${\cite{CFKZ}}$
\begin{eqnarray}
S_{t}f(x)=&&E^{0}_{x}[f(X^{0}_{t})\exp(\int_{0}^{t}(A^{-1}B)^{*}(X^{0}_{s})dM^{0}_{s}+(\int_{0}^{t}(A^{-1}\hat {B})^{*}(X^{0}_{s})dM^{0}_{s})\circ \gamma^{0}_{t}\nonumber\\
&&{}-\frac{1}{2}\int_{0}^{t}(B-\hat{B})A^{-1}(B-\hat{B})^{*}(X^{0}_{s})ds                                                                                                                                                                                                                                                                                                                                                                                                                                                                                                                                                                                                                                                                                                                                                                                                                                                                                                                                                                                                                                                                                                                                                                                                                                                                                                                                                                                                                                                                                                                                                                                                               + \int_{0}^{t}Q(X^{0}_{s})ds)]
\end{eqnarray}
$E^{0}_{x}$ denotes the expectation under $P^{0}_{x}$.\\
Set
\begin{eqnarray}
\hat{Z}_{t}&=&\exp(\int_{0}^{t}(A^{-1}B)^{*}(X^{0}_{s})dM^{0}_{s}+(\int_{0}^{t}(A^{-1}\hat {B})^{*}(X^{0}_{s})dM^{0}_{s})\circ \gamma^{0}_{t}\nonumber\\
&&{}-\frac{1}{2}\int_{0}^{t}(B-\hat{B})A^{-1}(B-\hat{B})^{*}(X^{0}_{s})ds                                                                                                                                                                                                                                                                                                                                                                                                                                                                                                                                                                                                                                                                                                                                                                                                                                                                                                                                                                                                                                                                                                                                                                                                                                                                                                                                                                                                                                                                                                                                                                                                              + \int_{0}^{t}Q(X^{0}_{s})ds).
\end{eqnarray}

By ${\cite{CZ}}$ and $\cite{YZ}$, there exists a bounded, continuous functions $v\in W^{1,p}(D)$ satisfying that
\begin{eqnarray}
\label{function v}
&&(\int_{0}^{t}(A^{-1}\hat {B})^{*}(X^{0}_{s})dM^{0}_{s})\circ \gamma^{0}_{t}\nonumber\\
&=&-\int_{0}^{t}\nabla v (X^{0}_{s})dM_{s}+v(X^{0}_{t})-v(X^{0}_{0})-\int_{0}^{t}(A^{-1}\hat {B})^{*}(X^{0}_{s})dM_{s}
\end{eqnarray}
Moreover, $v$ satisfies the following equations: for $g\in C^{1}(\bar{D})$,
\begin{eqnarray}
\int_{D}<A\nabla v,\nabla g>(x)dx=\int_{D}<\hat{B},\nabla g>(x)dx.
\end{eqnarray}
Thus the representation of $S_{t}$ becomes:
\begin{eqnarray}
\label{relationship between the two semigroups}
S_{t}f(x)
&=&e^{-v(x)}E^{0}_{x}[f(X^{0}_{t})e^{v(X^{0}_{t})}\exp(\int_{0}^{t}(A^{-1}(B-\hat B- A\nabla v))^{*}dM^{0}_{s}\nonumber\\
&&-\frac{1}{2}\int_{0}^{t}(B-\hat B- A\nabla v)^{*}A^{-1}(B-\hat B- A\nabla v)(X^{0}_{s})ds\nonumber\\
&&+\int_{0}^{t}(Q+\frac{1}{2}(\nabla v)A(\nabla v)^{*}-\langle B-\hat B, \nabla v\rangle)(X^{0}_{s})ds)]\nonumber\\
&=&e^{-v(x)}\tilde{S}_{t}[fe^{v}](x).
\end{eqnarray}
Here, setting
$b:=B-\hat{B}-(A\nabla v)$ and $q:=Q+\frac{1}{2}(\nabla v)A(\nabla v)^{*}-\langle B-\hat B, \nabla v\rangle$, we see that $\tilde{S}_{t}$ is the semigroup generated by the following operator:
\begin{eqnarray}
L_{2}&=&\frac{1}{2}\nabla\cdot(A\nabla)+(B-\hat{B}-(A\nabla v))\cdot\nabla+(Q+\frac{1}{2}(\nabla v)A(\nabla v)^{*}-\langle B-\hat B, \nabla v\rangle)\nonumber\\
&=£º&\frac{1}{2}\nabla\cdot(A\nabla)+b\cdot\nabla+q\nonumber
\end{eqnarray}
equipped with the boundary condition $\frac{\partial}{\partial \gamma}=0$.\\\\
In this section, we will stick to this particular choice of $b$ and $q$.\\\\
Recall that
$$
\tilde{M}(t)=e^{\int^{t}_{0}A^{-1}b(X^{0}_{s})dM^{0}_{s}-\frac{1}{2}\int_{0}^{t}bA^{-1}b^{*}(X^{0}_{s})ds}
$$
and set $Z_{t}=\tilde{M}(t)e^{\int_{0}^{t}q(X^{0}_{s})ds}$.\\
Then from $(\ref{function v})$, it follows that
$
\hat{Z}(t)=Z_{t}e^{v(X^{0}_{t})-v(X^{0}_{0})}.
$\\
Recall the operator $L_{1}=\frac{1}{2}\nabla\cdot(A\nabla)+b\cdot\nabla$ with Neumann boundary condition, which is associated with the reflecting diffusion $(X(t),P_{x})$. It is known from $\cite{LZ}$ that
\begin{eqnarray}
dP_{x}|_{\mathcal{F}_{t}}=\tilde{M}_{t}dP^{0}_{x}|_{\mathcal{F}_{t}},\nonumber
\end{eqnarray}
and
\begin{eqnarray}
X(t)=x+\int_{0}^{t}\sigma(X(s))dW_{s}+\int_{0}^{t}(\frac{1}{2}\nabla A+b)(X(s))ds+\int_{0}^{t}{\gamma}(X(s))dL_{s},\quad P_{x}-a.s.\nonumber
\end{eqnarray}
where $\{W_{t}\}$ is a d-dimensional Brownian motion and $L_{t}$ is the local time satisfying that
$L_{t}=\int_{0}^{t}I_{\partial D}(X(s))dL_{s}$.\\
\begin{lemma}
\label{lemma integrable condition dealing with 2q}
Assume that there exists $x_{0}\in D$, such that
\begin{eqnarray}
\label{integrable condition dealing with  2q}
E^{0}_{x_{0}}[\int_{0}^{\infty}|\hat{Z}_{t}|^{2}e^{\int_{0}^{t}(2Q-4r_{1})(X^{0}_{u})du}dL^{0}_{t}]<\infty.
\end{eqnarray}
Then there exists a positive number $\varepsilon>0$, if $\|\hat{B}\|_{L^{p}}\leq \varepsilon$, the following inequality holds:
\begin{eqnarray}
\label{integrability of 2q}
\sup_{x\in D}E_{x}[\int_{0}^{\infty}e^{2\int_{0}^{t}(-r_{1}+q)(X(u))du}dt]<\infty.
\end{eqnarray}
\end{lemma}
\begin{proof}
\begin{eqnarray}
E_{x}[e^{2\int_{0}^{t}(-r_{1}+q)(X(u))du}]&=&E^{0}_{x}[\tilde{M}(t)e^{2\int_{0}^{t}(-r_{1}+q)(X^{0}_{u})du}]\nonumber\\
&=&E^{0}_{x}[Z(t)e^{\int_{0}^{t}(-2r_{1}+q)(X(u))du}]\nonumber\\
&\leq&C_{1}E^{0}_{x}[\hat{Z}(t)e^{-2\int_{0}^{t}(r_{1}(X(u))du}e^{\int_{0}^{t}(Q+\frac{1}{2}<A\nabla v-2(B-\hat{B}),\nabla v>)(X_{u}^{0})du}]\nonumber\\
&\leq& C_{1}E^{0}_{x}[\hat{Z}^{2}(t)e^{2\int_{0}^{t}(Q-2r_{1})(X_{u}^{0})du}]^{\frac{1}{2}}\cdot
E^{0}_{x}[e^{\int_{0}^{t}<A\nabla v-2(B-\hat{B}),\nabla v>(X_{u}^{0})du}]^{\frac{1}{2}}\nonumber
\end{eqnarray}
By Lemma 3.3 and condition $(\ref{integrable condition dealing with  2q})$, there exists two constant $c_{2},\beta>0$ such that
$$
\sup_{x\in D}E_{x}[\hat{Z}^{2}(t)e^{2\int_{0}^{t}(Q-r_{1})(X_{u}^{0})du}]<c_{2}e^{-\beta t}.
$$
Moreover, for $p>d$, by the Theorem 2.1 in $\cite{LLZ}$, there exist two positive constants $c_{3}$ and $c_{4}$ such that
$$
E^{0}_{x}[e^{\int_{0}^{t}<A\nabla v-2(B-\hat{B}),\nabla v>(X_{u}^{0})du}]\leq c_{3}e^{c_{4}t},
$$
where $c_{4}=c\|<A\nabla v-2(B-\hat{B}),\nabla v>\|_{L^{p/2}}$. \\
Since $|\nabla v|_{L^{p}}\leq C|\hat{B}|_{L^{p}(D)}$ (see $\cite{YZ}$), there exists $\varepsilon>0$, such that $|\hat{B}|_{L^{p}(D)}\leq \varepsilon$ implies $c_{4}<\beta$. Thus
$(\ref{integrability of 2q})$ holds. $\square$
\end{proof}
\begin{theorem}
Assume $(\ref{integrable condition dealing with  2q})$  and for some point $x_{0}\in D$
\begin{eqnarray}
\label{condition of last theorem}
E^{0}_{x_{0}}[\int_{0}^{\infty}\hat{Z}_{s}dL^{0}_{t}]<\infty
\end{eqnarray}
Then there exists $\varepsilon>0$ such that if $\|\hat{B}\|_{L^{p}}\leq \varepsilon$, the problem $(\ref{final equation})$ has a unique, bounded, continuous weak solution $u(x)$.
\end{theorem}
\begin{proof}\\
Existence: Set $\tilde{F}(x,y)=e^{v(x)}F(x,e^{-v(x)}y)$ and $\phi(x)=e^{v(x)}\Phi(x)$.\\
From the boundedness of $v$ , $\tilde{F}$ is also bounded.\\
And $\tilde{F}$  satisfies
$$
(y_{1}-y_{2})(\tilde{F}(x,y_{1})-\tilde{F}(x,y_{2}))\leq-r_{1}(x)|y_{1}-y_{2}|^{2}.
$$
Moreover, there is a constant $c>0$, such that
\begin{eqnarray}
\infty>E^{0}_{x_{0}}[\int_{0}^{\infty}\hat{Z}_{s}dL^{0}_{s}]&=&
E^{0}_{x_{0}}[\int_{0}^{\infty}Z_{s}e^{v(X^{0}_{s})-v(X^{0}_{0})}dL^{0}_{s}]\nonumber\\
&\geq& cE^{0}_{x_{0}}[\int_{0}^{\infty}Z_{s}dL^{0}_{s}]=cE^{0}_{x_{0}}[\int_{0}^{\infty}\tilde{M}_{s}e^{\int_{0}^{s}q(X_{u}^{0})du}dL^{0}_{s}]
\end{eqnarray}
By Lemma $\ref{relationship between the integrale w.r.t. local times}$, we know that, at $x_{0}\in D$,
\begin{eqnarray}
\label{previous condition}
E_{x_{0}}[\int_{0}^{\infty}e^{\int_{0}^{s}q(X_{u})du}dL_{s}]<\infty.
\end{eqnarray}
Furthermore, by Lemma $\ref{boundedness of gauge function}$, it follows that
\begin{eqnarray}
\label{second conclusion}
\sup_{x}E_{x}[\int_{0}^{\infty}e^{\int_{0}^{t}q(X(u))du}dL_{t}]<\infty.
\end{eqnarray}
By Lemma $\ref{lemma integrable condition dealing with 2q}$, the following condition is satisfied :
\begin{eqnarray}
\label{first conclusion}
E_{x}[\int_{0}^{\infty}e^{2\int_{0}^{t}(q-r_{1})(X(u))du}dt]<\infty,
\end{eqnarray}

So $\tilde{F}$ satisfies all of the conditions in Theorem $\ref{semilinear equation L2}$ replacing $G$ by $\tilde{F}$.
Thus the following problem
\begin{eqnarray}
\label{simplified equation}
\left\{\begin{array}{ll}
L_{2}u(x)=-\tilde{F}(x,u(x)),&\textrm{on $D$}\\
\frac{1}{2}\frac{\partial u}{\partial \gamma}(x)=\phi &\textrm{on $\partial D$ }
\end{array}\right.
\end{eqnarray}
has a unique bounded continuous weak solution $u(x)$.\\
Set $f(x)=e^{-v(x)}u(x)$. Then we claim the function $f(x)$ is the weak solution of the equation $(\ref{final equation})$.\\
Because function $v$ is continuous and bounded, $f(x)$ is also continuous. From the fact that function $u$ is the weak solution of the problem $(\ref{simplified equation})$, we obtain, for any function $\psi\in C^{\infty}(D)$,
\begin{eqnarray}
\label{traslation between G and L}
&&\mathcal{E}(u,e^{-v}\psi)=\frac{1}{2}\int_{D}<A\nabla u,\nabla(e^{-v}\psi)>-<b,\nabla u>e^{-v}\psi-e^{-v}qu\psi dx\nonumber\\
&=&\int_{\partial D}e^{-v}\phi\psi d\sigma+\int_{D}\tilde{F}(x,u(x))e^{-v}\psi dx.
\end{eqnarray}
As in the proof of Theorem 5.1 in $\cite{Z}$, we can show that the left side of the equation $(\ref{traslation between G and L})$ equals to
\begin{eqnarray}
\mathcal{Q}(f,\psi)=\frac{1}{2}\int_{D}[<A\nabla f,\nabla \psi>-<B,\nabla u>\psi-<\hat{B},\nabla \psi>f-Qf\psi] dx.\nonumber
\end{eqnarray}
At the same time, by the definition of the function $\phi$ and $\tilde{F}$, the right side of the equation $(\ref{traslation between G and L})$ equals to
\begin{eqnarray}
\int_{\partial D}\Phi\psi d\sigma+\int_{D}{F}(x,f(x))\psi dx.\nonumber
\end{eqnarray}
Thus it follows that, for any $\psi\in C^{\infty}(D)$,
\begin{eqnarray}
\mathcal{Q}(f,\psi)=\int_{\partial D}\Phi\psi d\sigma+\int_{D}{F}(x,f(x))\psi dx.\nonumber
\end{eqnarray}
which proves that function $f$ is a weak solution of the problem $(\ref{final equation})$.\\
Uniqueness: \\
If $\bar{f}$ is another solution of the problem $(\ref{final equation})$, then $\bar{u}:=e^{v}f$ can be shown to be the solution of the equation $(\ref{simplified equation})$. Then by the uniqueness of the problem $(\ref{simplified equation})$ proved in the Theorem $\ref{semilinear equation L2}$, we find $\bar u=u$. Therefore, $f=\bar{f}$.$\square$
\end{proof}


\section{$L^{1}$  solutions of the BSDE and Semilinear PDEs}
\setcounter{equation}{0}
Recall the operator
\begin{eqnarray}
L_{1}=\frac{1}{2}\sum_{i,j=1}^d \frac{\partial}{\partial x_i}\left(a_{ij}(x)\frac{\partial}{\partial x_j}\right)+\sum_{i=1}^d  b_i(x)\frac{\partial}{\partial x_i}\nonumber
\end{eqnarray}
on the domian D equipped with the Neumann boundary condition $\frac{\partial}{\partial \gamma}=0$, on $\partial D$.\\
And  $(\Omega, \mathcal{F}_{t}, X(t), P_{x},x\in D)$ is the reflecting diffusion process associated with the generator $L_{1}$.\\
Then the process $X(t)$ has the following decomposition:
\begin{eqnarray}
X(t)=X(0)+M(t)+\int_{0}^{t}\tilde{b}(X(s))ds+\int_{0}^{t}An(X(s))dL_{s},\quad P_{x}-a.s..\nonumber
\end{eqnarray}
Here $\tilde{b}=\{\tilde{b}_{1},...,\tilde{b}_{d}\}$ with $\tilde{b}_{i}=\frac{1}{2}\sum_{j}\frac{\partial a_{ij}}{\partial x_{j}}+b_{i}$. $M(t)$ is the $\mathcal{F}_{t}$ square integrable continuous martingale additive functional.\\\\
In this section, we will consider the $L^{1}$ solutions of the BSDEs in Section 2 and use this result  to  solve the nonlinear elliptic partial differential equation with the mixed boundary condition.\\\\
Let $f:\Omega\times R^{+}\times R\rightarrow R$ be progressively measurable. Consider the following conditions:\\
\textbf{(I.1)} $(y-y')(f(t,y)-f(t,y'))\leq d(t)|y-y'|^{2}$, where $d(t)$ is a progressively measurable process;\\
\textbf{(I.2)}$E[\int_{0}^{\infty}e^{\int_{0}^{s}d(u)du}|f(s,0)|ds]<\infty$;\\
\textbf{(I.3)} $P_{x}-a.s.$, for any $t>0$, $y\rightarrow f(t,y)$ is continuous;\\
\textbf{(I.4)} $\forall r>0,\quad T>0$, $\displaystyle{\psi_{r}(t):=\sup_{|y|\leq r}|f(t,y)-f(t,0)|}\in L^{1}([0,T]\times \Omega, dt\times dP_{x}).$\\\\
The following lemma is deduced from Corollary 2.3 in $\cite{BDHPS}$.
\begin{lemma}
\label{lemma L1 estimate}
Suppose a pair of progressively measurable processes $(Y,Z)$ with values in $R\times R^{d}$ such that
$t\rightarrow Z_{t}$ belongs to $L^{2}([0,T])$ and $t\rightarrow f(t,Y_{t})$ belongs to $L^{1}([0,T])$,  $P_{x}-a.s.$.\\
If
\begin{eqnarray}
\label{BSDE}
Y_{t}=\xi+\int_{t}^{T}f(r,Y_{r})dr-\int_{t}^{T}<Z_{r},dM_{r}>,
\end{eqnarray}
then the following inequality holds, for $0\leq t< u\leq T$,
\begin{eqnarray}
|Y_{t}|\leq |Y_{u}|+\int_{t}^{u}\hat{Y}_{s}f(s,Y_{s})ds-\int_{t}^{u}\hat{Y}_{s}\langle Z_{r},dM_{r}\rangle.\nonumber
\end{eqnarray}
where $\hat{y}=\frac{y}{|y|}I_{\{y\neq 0\}}$.
\end{lemma}
The following lemma can be proved by modifying the proof of Proposition 6.4 in $\cite{BDHPS}$.
\begin{lemma}
\label{L1 BSDE with finite terminal time}
Assume that conditions (I.1)-(I.4) with $d(t)\equiv 0$. Then there exists a unique solution $(Y,Z)$ of the BSDE
\begin{eqnarray}
Y_{t}=\int_{t}^{T}f(r,Y_{r})dr-\int_{t}^{T}\langle Z_{r},dM_{r}\rangle, \quad for \quad t\leq T.
\end{eqnarray}
Moreover, for each $\beta\in(0,1)$, $E[\sup_{t\leq T}|Y_{t}|^{\beta}]+E[(\int_{0}^{T}|Z_{r}|^{2}dr)^{\frac{\beta}{2}}]<\infty$.
\end{lemma}
\vspace{4mm}
Suppose $\beta\in(0,1)$. \\
$\mathcal{S}^{\beta}$ denotes the set of real-valued, adapted and continuous process $\{Y_{t}\}_{t\geq 0}$ such that
$$
\|Y\|^{\beta}:=E[\sup_{t>0}|Y_{t}|^{\beta}]<\infty.
$$
It is known that $\|\cdot\|^{\beta}$ deduces a complete metric on $\mathcal{S}^{\beta}$. \\
$M^{\beta}$ denotes the set of $R^{d}$-valued predictable processes $\{Z_{t}\}$ such that
$$
\|Z\|_{M^{\beta}}:=E[(\int_{0}^{\infty}|Z_{t}|^{2}dt)^\frac{\beta}{2}]<\infty.
$$
$M^{\beta}$ is also a complete metric space with the distance deduced by $\|\cdot\|_{M^{\beta}}$.
\begin{lemma}
\label{lemmad(t)is0infinite}
Under the same assumption as the Lemma $\ref{L1 BSDE with finite terminal time}$,  there exists a unique solution $(Y,Z)$ of the BSDE
\begin{eqnarray}
\label{d(t)is0infinite}
&&Y_{t}=Y_{T}+\int_{t}^{T}f(r,Y_{r})dr-\int_{t}^{T}\langle Z_{r},dM_{r}\rangle,\quad any\quad t\leq T;\nonumber\\
&&\lim_{t\rightarrow\infty}Y_{t}=0,\quad P-a.s..
\end{eqnarray}
\end{lemma}
\begin{proof}
Existence:\\
By the Lemma 6.2 above, there exists  $(Y^{n},Z^{n})$ such that, for $0\leq t\leq n$,
$$
Y^{n}_{t}=\int_{t}^{n}f(r,Y^{n}_{r})dr-\int_{t}^{n}\langle Z^{n}_{r},dM_{r}\rangle,
$$
and $Y^{n}_{t}=Z^{n}_{t}=0$, for $t\geq n$.\\
Fix $t>0$ and $t<n<n+i$, then
\begin{eqnarray}
Y^{n+i}_{t}-Y^{n}_{t}=\int_{t}^{n+i}(f(r,Y^{n+i}_{r})-f(r,Y^{n}_{r}))dr-\int_{t}^{n+i}\langle(Z^{n+i}_{r}-Z^{n}_{r}),dM_{r}\rangle
+\int_{n}^{n+i}f(r,0)dr\nonumber
\end{eqnarray}
Set $F^{n}(r,y)=f(r,y+Y^{n}_{r})-f(r,Y^{n}_{r})+f(r,0)I_{\{r>n\}}$, $y^{n}_{t}=Y^{n+i}_{t}-Y^{n}_{t}$ and $z^{n}_{t}=Z^{n+i}_{t}-Z^{n}_{t}$. Then $(y^{n}_{t},z^{n}_{t})$ is the solution of the following BSDE:
\begin{eqnarray}
\label{d(t)is0infinite: difference equation}
y^{n}_{t}=\int_{t}^{n+i}F(r,y^{n}_{r})dr-\int_{t}^{n+i}\langle z^{n}_{r},dM_{r}\rangle.
\end{eqnarray}
So that by the condition (I.1) with $d(t)\equiv 0$, it follows from Lemma $\ref{lemma L1 estimate}$ that
\begin{eqnarray}
|y^{n}_{t}|&\leq&\int_{t}^{n+i}\langle\hat{y}^{n}_{r},F^{n}(r,y^{n}_{r})\rangle dr-\int_{t}^{n+i}\langle \hat{y}^{n}_{r},z^{n}_{r}dM_{r}\rangle\nonumber\\
&\leq& \int_{t}^{n+i}\frac{I_{\{y^{n}_{r}\neq 0\}}}{|y^{n}_{r}|}\langle y^{n}_{r},f(r,y^{n}_{r}+Y^{n}_{r})-f(r,Y^{n}_{r})\rangle dr+\int_{n}^{n+i}|f(s,0)|ds\nonumber\\
&&-
\int_{t}^{n+i}\langle \hat{y}^{n}_{r},
z^{n}_{r}dM_{r}\rangle\nonumber\\
&\leq&\int_{n}^{n+i}|f(s,0)|ds-\int_{t}^{n+i}\langle \hat{y}^{n}_{r},z^{n}_{r}dM_{r}\rangle.
\end{eqnarray}
Taking conditional expectation on both side of the inequality, we got
$$
|y^{n}_{t}|\leq E[\int_{n}^{n+i}|f(s,0)|ds|\mathcal{F}_{t}]:=M^{n}_{t},
$$
where $M^{n}_{t}$ is a martingale. Then by Doob's inequality and condition (I.2), it follows that, for $\beta\in (0,1)$,
\begin{eqnarray}
E[\sup_{t}|y^{n}_{t}|^{\beta}]\leq E[\sup_{t}(M^{n}_{t})^{\beta}]&\leq& \frac{1}{1-\beta}E[\int_{n}^{n+i}|f(s,0)|ds]^{\beta}\nonumber\\
&\rightarrow& 0,\quad as \quad n\rightarrow\infty.
\end{eqnarray}
Therefore, $\{Y^{n}\}$ is a Cauchy sequence under the norm $\|\cdot\|^{\beta}_{\infty}$. So that there is a process $Y$ such that $E[\sup_{t}|Y_{t}-Y^{n}_{t}|^{\beta}]\rightarrow 0$.\\
This also implies that $Y_{t}\rightarrow 0$, as $t\rightarrow \infty$, $P_{x}-a.s.$.\\
Moreover, by the equation $(\ref{d(t)is0infinite: difference equation})$, Ito's formula and the condition (I.1), it follows that
\begin{eqnarray}
&&|y^{n}_{t}|^{2}+\int_{t}^{n+i}\langle A(X(r))z^{n}_{r},z^{n}_{r}\rangle dr\nonumber\\
&=&2\int_{t}^{n+i}\langle y^{n}_{r},F^{n}(r,y^{n}_{r})\rangle dr-2\int_{t}^{n+i}\langle y^{n}_{r},z^{n}_{r}dM_{r}\rangle \nonumber\\
&\leq&2\int_{n}^{n+i}\langle y^{n}_{r},f^{n}(r,0)\rangle dr+2|\int_{t}^{n+i}\langle y^{n}_{r},z^{n}_{r}dM_{r}\rangle |\nonumber\\
&\leq&\sup_{r}|y^{n}_{r}|^{2}+(\int_{n}^{n+i}|f(r,0)|dr)^{2}+2|\int_{t}^{\infty}\langle y^{n}_{r},z^{n}_{r}dM_{r}\rangle|,\nonumber
\end{eqnarray}
and thus that
\begin{eqnarray}
(\int_{t}^{n+i}|z^{n}_{r}|^{2}dr)^{\frac{\beta}{2}}\leq c_{1}[\sup_{r}|y^{n}_{r}|^{\beta}+(\int_{n}^{n+i}|f(r,0)|dr)^{\beta}+|\int_{t}^{n+i}<y^{n}_{r},z^{n}_{r}dM_{r}>|^{\frac{\beta}{2}}].\nonumber
\end{eqnarray}
Taking expectation on both sides of the inequality and applying the BDG inequality, we obtain
\begin{eqnarray}
&&E[(\int_{t}^{n+i}|z^{n}_{r}|^{2}dr)^{\frac{\beta}{2}}]\nonumber\\
&\leq&
c_{1}(E[\sup_{r}|y^{n}_{r}|^{\beta}]+E[(\int_{t}^{n+i}|f(r,0)|dr)^{\beta}])
+c_{2}E[(\int_{t}^{n+i}|y^{n}_{r}|^{2}|z^{n}_{r}|^{2}dr)^{\frac{\beta}{4}}]\nonumber\\
&\leq&c_{1}(E[\sup_{r}|y^{n}_{r}|^{\beta}]+E[(\int_{n}^{n+i}|f(r,0)|dr)^{\beta}])+
c_{2}E[(\sup_{r}|y^{n}_{r}|^{\frac{\beta}{2}}\int_{t}^{n+i}|z^{n}_{r}|^{2}dr)^{\frac{\beta}{4}}]\nonumber\\
&\leq&(c_{1}+\frac{c_{2}}{2})(E[\sup_{r}|y^{n}_{r}|^{\beta}]+E[(\int_{n}^{n+i}|f(r,0)|dr)^{\beta}])
+\frac{1}{2}E[(\int_{t}^{n+i}|z^{n}_{r}|^{2}dr)^{\frac{\beta}{2}}].\nonumber
\end{eqnarray}
Therefore, we know that there is a constant $C>0$, such that
\begin{eqnarray}
E[(\int_{0}^{\infty}|z^{n}_{s}|^{2}ds)^{\frac{\beta}{2}}]&\leq& CE[\sup_{t}|y^{n}_{t}|^{\beta}+(\int_{n}^{n+i}|f(s,0)|ds)^{\beta}]\nonumber\\
&\leq& CE[\sup_{t}|y^{n}_{t}|^{\beta}]+CE[\int_{n}^{n+i}|f(s,0)|ds]^{\beta}\nonumber\\
&\rightarrow& 0\quad as\quad n\rightarrow\infty.\nonumber
\end{eqnarray}
So that $\{Z^{n}_{t}\}$ is a Cauchy sequence in $M^{\beta}$. Let $Z$ denote the limit of $\{Z^{n}\}$. \\
At last, by the condition (I.3), we find that
\begin{eqnarray}
\label{d(t)is0infinite:function f convergence}
\int_{0}^{T}f(t,Y^{n}_{t})dt\rightarrow\int_{0}^{T}f(t,Y_{t})dt,\quad P_{x}-a.s..
\end{eqnarray}
Therefore, (Y,Z) is the solution satisfies the BSDE $(\ref{d(t)is0infinite})$.\\
Uniqueness:\\
Consider $(Y,Z)$ and $(Y',Z')$ are two solutions to $(\ref{d(t)is0infinite})$. Then by the same method as in the proof of Lemma 2.1, we can show that,
\begin{eqnarray}
\forall t>0,\quad |Y_{t}-Y'_{t}|=0,\quad P-a.s..\nonumber\quad \quad \quad\quad \quad \quad \square\nonumber
\end{eqnarray}
\end{proof}
\textbf{(I.5)}  The process $d(t)$ is a progressively measurable process satisfying
$$
d(\cdot)\in L^{1}[[0,T]\times\Omega,dt\otimes P], \quad for \quad any \quad T>0.
$$
\begin{theorem}
\label{thoerem infinite horizon}
Assume the conditions (I.1)-(I.4). Then there exists a unique process $(Y,Z)$ such that,
\begin{eqnarray}
\label{infinite horizon BSDE}
&&Y_{t}=Y_{T}+\int_{t}^{T}f(r,Y_{r})dr-\int_{t}^{T}<Z_{r},dM_{r}>,\quad for \quad any \quad t<T;\nonumber\\
&&\lim_{t\rightarrow\infty}e^{\int_{0}^{t}d(u)du}Y_{t}=0,\quad P-a.s.
\end{eqnarray}
\end{theorem}
\begin{proof}\\
Existence:\\
Set $\hat{f}(t,y)=e^{\int_{0}^{t}d(u)du}f(t,e^{-\int_{0}^{t}d(u)du}y)-d(t)y$. Then \\
(1) $(y-y')(\hat{f}(t,y)-\hat{f}(t,y'))\leq 0$;\\
(2)$\hat{f}(t,0)=e^{\int_{0}^{t}d(u)du}f(t,0)$. So
$
E[\int_{0}^{\infty}|\hat{f}(s,0)|ds]=E[\int_{0}^{\infty}e^{\int_{0}^{t}d(u)du}|f(t,0)|ds]<\infty.\\
$
(3)$\displaystyle{\sup_{|y|\leq r}}|\hat{f}(t,y)-\hat{f}(t,0)|\leq \psi_{r}(t)+|d(t)|r$, where the process $\psi_{r}(t)+|d(t)|r\in L^{1}([0,T]\times \Omega,dt\otimes P)$, for $T>0$.\\
Therefore, $\hat{f}$ satisfies all the conditions of the Lemma $\ref{lemmad(t)is0infinite}$. So there exists a pair of processes $(\hat{Y},\hat{Z})$ satisfying the equation:
$$
\hat{Y}_{t}=\hat{Y}_{T}+\int_{t}^{T}\hat{f}(r,\hat{Y}_{r})dr-\int_{t}^{T}\langle \hat{Z}_{r},dM_{r}\rangle,
$$
and obviously $\displaystyle{\lim_{t\rightarrow\infty}\hat{Y}_{t}=0}$.\\
By the chain rule and the definition of the function $\hat{f}$, it follows that
\begin{eqnarray}
de^{-\int_{0}^{t}d(u)du}\hat{Y}_{t}=-f(t,e^{-\int_{0}^{t}d(u)du}\hat{Y}_{t})dt+\langle e^{-\int_{0}^{t}d(u)du}\hat{Z}_{t},dM_{t}\rangle.\nonumber
\end{eqnarray}
Set $Y_{t}=e^{-\int_{0}^{t}d(u)du}\hat{Y}_{t}$ and $Z_{t}=e^{-\int_{0}^{t}d(u)du}\hat{Z}_{t}$. Then the process $(Y,Z)$ is the solution to the equation $(\ref{infinite horizon BSDE})$.\\
Uniqueness:\\
The uniqueness of the solution to $(\ref{infinite horizon BSDE})$ follows from the uniqueness of the solution to equation $(\ref{d(t)is0infinite})$. $\square$
\end{proof}\\
\\

Let $G(x,y):R^{d}\times R\rightarrow R$ be a bounded Borel measurable function.  Consider the following conditions:\\
$\textbf{(H.1)}^{'}$ $(y_{1}-y_{2})(G(x,y_{1},z)-G(x,y_{2},z))\leq -h_{1}(x)|y_{1}-y_{2}|^{2}$, where $h_{1}\in L^{p}(D)$ for $p>\frac{d}{2}$.\\
$\textbf{(H.2)}^{'}$   $y\rightarrow G(x,y)$ is continuous.\\
\begin{theorem}
\label{theorem: operator G}
Assume the Conditions $(H.1)^{'}$ and $(H.2)^{'}$  and that
there is some point $x_{0}\in D$, such that
\begin{eqnarray}
\label{L1:condition of the theorem about operator G}
E_{x_{0}}[\int_{0}^{\infty}e^{\int_{0}^{s}q(X(u))du}dL_{s}]<\infty.
\end{eqnarray}
Then the semilinear Neumann boundary value problem
\begin{eqnarray}
\left\{\begin{array}{ll}
{L_{2}}u(x)=-G(x,u(x)),&\textrm{on $D$}\\
\frac{\partial u}{\partial{\gamma}}(x)=\phi(x) &\textrm{on $\partial D$ }
\end{array}\right.
\end{eqnarray}
has a unique continuous weak solution.
\end{theorem}
\begin{proof}\\
Step 1 \\
Set
$
\tilde{G}(X(t),y)=e^{\int_{0}^{t}q(X(u))dt}G(x,e^{-\int_{0}^{t}q(X(u))dt}y).
$
Then there exists a unique solution $(\hat{Y}_{x},\hat{Z}_{x})$ to the following BSDE:\\
for any $T>0$ and $0<t<T$,
\begin{eqnarray}
\hat{Y}_{x}(t)&=&\hat{Y}_{x}(T)+\int_{t}^{T}\tilde{G}(X(t),\hat{Y}_{x}(s))ds-\int_{t}^{T}e^{\int_{0}^{s}{q}(X(u))dt}\phi(X(s))dL_{s}\nonumber\\
&&-\int_{t}^{T}\langle \hat{Z}_{x}(s),dM_{x}(s)\rangle\nonumber
\end{eqnarray}
and
\begin{eqnarray}
\lim_{t\rightarrow \infty}e^{-\int_{0}^{t}h_{1}(X(u))du}\hat{Y}_{t}=0 \quad P_{x}-a.s.\nonumber
\end{eqnarray}
The uniqueness follows from the uniqueness proved in  Theorem 6.1.
Only the existence of solution $(\hat{Y}_{x},\hat{Z}_{x})$ needs to be proved:\\
(a) Similarly as the proof of Theorem 2.1, we can show that there exists $(p_{x}(t),q_{x}(t))$ such that
\begin{eqnarray}
&&dp_{x}(t)=e^{\int_{0}^{t}{q}(X(u))du}\phi(X(t))dL_{t}+<q_{x}(t),dM_{x}(t)>,\nonumber\\
&&e^{-\int_{0}^{t}h_{1}(X(u))du}p_{x}(t)\rightarrow0, \quad as \quad t\rightarrow \infty, \quad P_{x}-a.s..
\end{eqnarray}
(b) Set $g(x,y)=\tilde{G}(x,y+p_{x}(t))$. Then it follows that\\
$$
\quad (y-y')(g(x,y)-g(x,y'))\leq -h_{1}(x)|y-y'|^{2}.
$$
The condition $(\ref{L1:condition of the theorem about operator G})$ and Lemma 3.3 imply, for $x\in D$,
$$
E_{x}[\int_{0}^{\infty}e^{\int_{0}^{s}(-h_{1}+q)(X(u))du}ds]<\infty.
$$
Furthermore, as the function $G$ is bounded, we see that condition $(I.2)$ is satisfied:
\begin{eqnarray}
&&E_{x}[\int_{0}^{\infty}e^{-\int_{0}^{s}h_{1}(X(u))du}|g(X(s),0)|ds]\nonumber\\
&=&E_{x}[\int_{0}^{\infty}e^{-\int_{0}^{s}h_{1}(X(u))du}|\tilde{G}(X(s),p_{x}(s))|ds]\nonumber\\
&=&E_{x}[\int_{0}^{\infty}e^{\int_{0}^{s}(-h_{1}+q)(X(u))du}|G(X(s),e^{-\int_{0}^{s}q(X(u))du}p_{x}(s))|ds]\nonumber\\
&\leq&\|G\|_{\infty}E_{x}[\int_{0}^{\infty}e^{\int_{0}^{t}(-h_{1}+q)(X(u))du}dt]\nonumber\\
&<&\infty.
\end{eqnarray}
Obviously  condition (I.3) is satisfied, i.e.,   $y\rightarrow g(x,y)$ is continuous.\\
Moreover, the condition (I.4) is also satisfied. In fact, for any $r>0$,
$$
\psi_{r}(t)=\sup_{r}|\tilde{G}(X(t),y)-\tilde{G}(X(t),0)|\leq 2\|G\|_{\infty}e^{\int_{0}^{t}q(X_{t})dt},
$$
and for any $T>0$, by the fact that $q\in L^{p}(D)$ with $p>\frac{d}{2}$ and Theorem 2.1 in $\cite{LLZ}$, $E_{x}[\int_{0}^{T}e^{\int_{0}^{t}q(X_{u})du}dt]<\infty$.\\
Therefore, the function $g(x,y)$ satisfies all of the conditions of Theorem $\ref{thoerem infinite horizon}$
. There exists a pair of processes $(y_{x}(t),z_{x}(t))$ such that
for any $T>0$ and $0<t<T$,
\begin{eqnarray}
y_{x}(t)=y_{x}(T)+\int_{t}^{T}g(X(s),y_{x}(s))ds-\int_{t}^{T}\langle z_{x}(s),dM_{x}(s)\rangle
\end{eqnarray}
and
\begin{eqnarray}
\lim_{t\rightarrow \infty}e^{-\int_{0}^{t}h_{1}(X(u))du}y_{x}(t)=0 \quad P_{x}-a.s.
\end{eqnarray}
Put $\hat{Y}_{x}(t)=p_{x}(t)+y_{x}(t)$ and $\hat{Z}_{x}(t)=q_{x}(t)+z_{x}(t)$. It follows that $(\hat{Y}_{x}(t),\hat{Z}_{x}(t))$ satisfies the following equation
$$
d\hat{Y}_{x}(t)=e^{\int_{0}^{t}{q}(X(u))du}\phi(X(t))dL_{t}-\tilde{G}(t,\hat{Y}_{x}(t))dt+<\hat{Z}_{x}(t),dM_{x}>,
$$
$$
\lim_{t\rightarrow \infty}e^{-\int_{0}^{t}h_{1}(X(u))du}\hat{Y}_{t}=0 \quad P_{x}-a.s..
$$
Step 2.\\
Put $Y_{x}(t):=e^{-\int_{0}^{t}q(X(u))dt}\hat{Y}_{x}(t)$ and $Z_{x}(t):=e^{-\int_{0}^{t}q(X(u))dt}\hat{Z}_{x}(t)$, we have
$$
dY_{x}(t)=-F(X(t),Y_{x}(t))+\phi(X(t))dL_{t}+<Z_{x}(t),dM_{x}(t)>,
$$
where $F(x,y)=q(x)y+G(x,y)$.
Moreover,
\begin{eqnarray}
e^{\int_{0}^{t}(-h_{1}+q)(X(u))du}Y_{x}(t)&=&e^{\int_{0}^{t}(-h_{1}+q)(X(u))(u)dt}e^{-\int_{0}^{t}q(X(u))dt}\hat{Y}_{x}(t)\nonumber\\
&=&e^{-\int_{0}^{t}h_{1}(X(u))dt}\hat{Y}_{x}(t)\rightarrow0\quad as\quad t\rightarrow\infty.\nonumber
\end{eqnarray}
Put $u_{0}(x)=Y_{x}(0)$ and $v_{0}(x)=Z_{x}(0)$. \\
Now as in the proof of Theorem, 4.1,  we can solve the following equation
\begin{eqnarray}
\left\{\begin{array}{ll}
{L_{2}}u(x)=-G(x,u_{0}(x)),&\textrm{on $D$}\\
\frac{1}{2}\frac{\partial u}{\partial{\gamma}}(x)=\phi(x) &\textrm{on $\partial D$ }
\end{array}\right.
\end{eqnarray}
and prove that the solution $u$ coincides with $u_{0}(x)$. This completes the proof of the  theorem. $\square$
\end{proof}
\\\\
Suppose that  $F:R^{d}\times R\rightarrow R$ is a bounded measurable function and $r_{1}\in L^{p}(D)$. Consider the following conditions :\\
\textbf{(E.1)} $(y_{1}-y_{2})(G(x,y_{1},z)-G(x,y_{2},z))\leq -r_{1}(x)|y_{1}-y_{2}|^{2}$;\\
\textbf{(E.3)}  $y\rightarrow F(x,y)$ is continuous;\\ \\
Now, after establishing Theorem 6.2,  following the same proof as that of Theorem 5.1, we finally have the following main result.\\
\begin{theorem}
Suppose that the function F satisfies the condition (E.1) and (E.2), and there exists $x_{0}\in D$ such that
\begin{eqnarray}
\label{L1:condition of the last thoerem}
E^{0}_{x_{0}}[\int_{0}^{\infty}\hat{Z}_{s}dL^{0}_{t}]<\infty.
\end{eqnarray}
Then the following problem
\begin{eqnarray}
\left\{\begin{array}{ll}
{L}u(x)=-F(x,u(x)),&\textrm{on $D$}\\
\frac{1}{2}\frac{\partial u}{\partial \gamma}(x)-<\widehat{B},n>(x)u(x)=\Phi(x) &\textrm{on $\partial D$ }
\end{array}\right.
\end{eqnarray}
has a unique, bounded, continuous weak solution.
\end{theorem}

\end{document}